%% file: main.tex
\documentclass[11pt]{article}
\usepackage{amsmath,amssymb,bbm,amsthm}
\usepackage{fullpage}
\usepackage{thm-restate,color,xspace}
\usepackage{graphicx}
\usepackage{subfiles} 
\usepackage[dvipsnames]{xcolor}
\usepackage{tikz}
\usepackage{tikz-cd}
\usepackage{mathrsfs}
\usepackage{environ}
\usepackage{todonotes}

\usepackage{amsmath}   
\usepackage{multirow}  
\usepackage{booktabs}

\usepackage[square,numbers,sort&compress]{natbib}
\usepackage{minted}
\usepackage{caption}
\usepackage{subcaption}

\usepackage{algorithm}
\usepackage{algpseudocode}

\usepackage{float}

\usepackage[shortlabels]{enumitem}

\usepackage[
    colorlinks=true,
    citecolor=blue,
    linkcolor=green!60!black,
    urlcolor=blue
]{hyperref} 
\usepackage{cleveref}
\usepackage{textgreek}

\usepackage{ulem} 
\usepackage{scalerel} 

\newcommand{\cI}{\mathcal I}

\newcommand{\R}{\mathbb{R}}

\newcommand{\ceil}[1]{\lceil#1\rceil}

\newtheorem{defn}{Definition}

\newtheorem{theorem}{Theorem}
\newtheorem{lemma}[theorem]{Lemma}
\newtheorem{definition}[theorem]{Definition}
\newtheorem{corollary}[theorem]{Corollary}

\newtheorem{remark}[theorem]{Remark}

\usepackage[title]{appendix}

\newcommand{\conv}{\text{\tt conv}}

\newcommand{\slope}{\text{\tt slope}}

\newcommand{\argmin}{\text{argmin}}
\newcommand{\argmax}{\text{argmax}}

\newcommand{\bitan}{\texttt{\tt bitan}}

\newcommand{\remove}[1]{}

\makeatletter
\newcommand{\plainfootnote}[1]{%
  \begingroup
  \renewcommand{\@makefnmark}{}
  \footnotetext{#1}%
  \endgroup
}
\makeatother

\title{Efficient Convexification of Kolmogorov-Arnold Networks with Polynomial Functional Forms Via a Continuous Graham Scan Approach
}
\author{Owen Li$^1$\textsuperscript{$\dagger$},  Daniel Ovalle$^2$\textsuperscript{$\dagger$}, Barnabas Poczos$^3$, Carl D. Laird$^2$, Ignacio E. Grossmann$^2$ \\ \& Javier Pe\~na$^4$\thanks{Corresponding author. Email: \texttt{jfp@andrew.cmu.edu}}\\ \\
    \small $^1$School of Operations Research and Information Engineering, Cornell University, Ithaca, NY.\\
    \small $^2$Department of Chemical Engineering, Carnegie Mellon University, Pittsburgh, PA.\\
    \small $^3$School of Computer Science, Carnegie Mellon University, Pittsburgh, PA.\\
    \small $^4$Tepper School of Business, Carnegie Mellon University, Pittsburgh, PA.\\}
\date{\today}

\begin{document}

\maketitle
\plainfootnote{$^\dagger$These authors contributed equally.}
\subfile{sections/abstract}
\subfile{sections/introduction}

\subfile{sections/background}

\subfile{sections/roots}
\subfile{sections/bitan_no_mark}
\subfile{sections/graham}

\subfile{sections/gam}
\subfile{sections/results_v2}


\bibliography{references}

\end{document}

%% file: sections/abstract.tex
\section*{Abstract}
\small
Deterministic global optimization of nonlinear models is important in many scientific and engineering applications. This framework typically involves repeatedly solving convex relaxations of the nonconvex problem, meaning that the strength of the relaxations and the cost of computing them directly determine overall efficiency and solution quality.
In this work, we develop a tailored continuous convexification framework for Kolmogorov–Arnold Networks in which the univariate components are polynomial functions. By exploiting the additive separable structure of this architecture, the relaxation problem reduces to computing tight convex envelopes of univariate polynomials. We propose a continuous variant of the classical Graham Scan that constructs these envelopes exactly by identifying the bitangents of the polynomial convex hull without discretization or factorable reformulations. We establish the correctness of the algorithm and characterize its computational complexity, and show how these envelopes can be combined to construct strong convex relaxations for polynomial KANs. Computational results demonstrate that the proposed relaxations are both strong and robust, often producing bounds that are comparable, or even orders of magnitude tighter than relaxations of state-of-the-art global optimization solvers while remaining computationally efficient.

%% file: sections/introduction.tex
\section{Introduction}

Deterministic global optimization focuses on solving nonconvex problems while providing guarantees of global optimality \cite{locatelli2021global}. A key component of these methods is the construction of convex relaxations, which replace the original nonconvex problem with a tractable convex approximation whose optimum provides a valid bound on the global solution \cite{ballerstein2014extended}. Convex relaxations are used in many algorithmic frameworks, including branch-and-bound methods \cite{csallner2000multisection,belotti2009branching,berenguel2013interval,fowkes2013branch,araya2016interval}, outer approximation \cite{duran1986outer,sherali1992global, anstreicher2010computable}, cutting-plane approaches \cite{bonami2019solving, santana2020convex, bienstock2020outer}, and other convexification-based algorithms \cite{khajavirad2012convex,wechsung2015reverse, kazazakis2018arbitrarily}. In all these settings, the quality of the relaxation strongly affects performance: tight relaxations provide stronger bounds, reduce the search space, and speed up convergence, while weak relaxations lead to poor bounds and higher computational cost \cite{tawarmalani2005polyhedral, tawarmalani2013convexification}. Developing convex relaxations that are both strong and efficient remains a central challenge in deterministic global optimization.

In many modern applications, optimization problems include surrogate models that approximate complex physical processes, simulations, or data-driven relationships \cite{bhosekar2018advances, schweidtmann2019deterministic, ovalle2025conformal}. These surrogates may come from machine learning models or reduced-order approximations of expensive simulations, allowing optimization without repeatedly running costly underlying models \cite{cozad2014learning}. Such approaches have been employed in diverse fields,  including process systems engineering \cite{zhang2016data}, energy systems design \cite{chatzivasileiadis2020decision}, healthcare \cite{bertsimas2016analytics}, and supply chain management \cite{ovalle2024integration}. Embedding these surrogates into a global optimization problem introduces nonlinearities and structural complexity that must be handled carefully. As a result, constructing convex relaxations that are both tight and computationally efficient is challenging, especially when the surrogate has highly nonlinear components or high-dimensional interactions \cite{parker2024formulations}. These difficulties motivate the use of surrogates with structures that can be exploited to produce stronger and more scalable relaxations.

Kolmogorov–Arnold Networks (KANs) have recently been proposed as an alternative neural network architecture that departs from the conventional multilayer perceptron paradigm \cite{liu2024kan}. Their design is motivated by the Kolmogorov–Arnold representation theorem \cite{kolmogorov1957representations}, which states that multivariate continuous functions can be represented as compositions of sums of univariate  functions \cite{schmidt2021kolmogorov}. In a KAN with $L$ layers, the model takes the form
\begin{equation}\label{eq:kan}
    \mathbf y = \left(\Phi_L \circ \Phi_{L-1} \circ \cdots \circ \Phi_1 \right)(\mathbf x)
\end{equation}
where each layer $\Phi_K: \R^{d_{K-1}} \to \R^{d_K}$ is defined component-wise as
\begin{equation}
(\Phi_K(\mathbf x))_i = \sum_j^{d_{K-1}} \phi_{K, i, j}(x_j),
\end{equation}
with $\phi_{K,i,j}$ denoting univariate functions. Unlike standard neural networks, which apply fixed nonlinear activations after linear transformations, KANs learn the nonlinear univariate components directly, with interactions between variables arising through layer composition. This makes the models flexible, often more interpretable, and highly separable \cite{somvanshi2025survey}. For global optimization, this structure is particularly useful, as it reduces the convexification challenge from general multivariate nonlinearities to constructing convex envelopes for the individual univariate functions.

In the original KAN formulation, the univariate functions $\phi_{k,i,j}$ are often represented using spline-based parameterizations, such as linear combinations of B-spline basis functions \cite{liu2024kan}. While flexible for function approximation, these representations are less convenient for deriving analytical convex relaxations due to their discrete nature \cite{karia2025deterministic}.  
In this work, we therefore focus on polynomial KANs, a relevant class of KANs in which each univariate component is represented by a polynomial, enabling both expressive modeling and tractable convexification.
Formally, we assume that each $\phi_{k,i,j}$ is a univariate polynomial of degree $n$. This restriction is well justified for several reasons. First, polynomials can approximate smooth functions arbitrarily well on compact domains when the degree is high enough, preserving the  expressive power of the model \cite{xue2010polynomial}. Second, polynomials have compact algebraic representations, which are convenient for optimization \cite{misener2023formulating}. Most importantly, the convex envelope of a polynomial over a bounded interval can be computed analytically, allowing tight relaxations to be derived directly from the structure of the model. As a result, polynomial KANs keep the benefits of the original architecture while supporting systematic construction of strong convex relaxations.

Despite the favorable structure of polynomial KANs, existing relaxation techniques remain limited. Factorable relaxations, such as McCormick envelopes \cite{mccormick1976computability}, fail to exploit the algebraic structure of polynomials and typically weaken as nonlinear compositions deepen. Exact convexification methods, including algebraic constructions, can provide tight envelopes but are computationally expensive and difficult to update efficiently when variable bounds change \cite{wright1996convex, katsamaki2023exact}. Similarly, B-spline representations, as used by \citet{karia2025deterministic}, require mixed-integer reformulations and rely on discrete enhancements for relaxation strength rather than the continuous polynomial structure. 
Therefore, the current approaches do not provide a practical mechanism for efficiently computing tight convex envelopes for polynomial KANs and motivate the development of specialized methods that leverage their algebraic structure.

To overcome the limitations of existing approaches, we propose the Continuous Graham Scan, a novel algorithm for computing the exact convex and concave envelopes of univariate polynomials. Operating directly on the continuous graph of the polynomial and inspired by the classical Graham Scan for convex hulls \cite{graham1972efficient}, the method avoids both discretization and factorable reformulations. It guarantees the exact convex envelope over any bounded interval, runs in polynomial time, and is applicable to polynomials of arbitrary degree. Because each layer of a polynomial KAN is additively composed of univariate functions, computing these envelopes for the individual components immediately produces tight relaxations for the full network. 
Moreover, for monotonic Generalized Additive Models (GAMs) (a practically relevant subclass of KANs consisting of additive components combined with a monotone, differentiable link function)
the proposed relaxation is exact at the global optimum when minimizing the model \cite{hastie1986generalized, hastie2017generalized}.
Finally, we present computational experiments showing that exploiting the separable structure of KAN layers produces relaxations that are both strong and robust, achieving comparable or better bounds than state-of-the-art solvers.

The remainder of the paper is organized as follows. Section \ref{sec:background} introduces the necessary background, assumptions, and notation used throughout the paper. In Section \ref{sect:graham}, we present the Continuous Graham Scan algorithm in detail, describing its construction, correctness, and computational properties. Section \ref{sec:pkan} shows how the dervied relaxation can be used to relax polynomial KANs and shows its tightness over GAMs. Section \ref{sec:results} reports numerical results, demonstrating the quality and computational efficiency of the proposed relaxations compared to existing state-of-the-art global optimization solvers, and Section \ref{sec:conclusions} concludes with a discussion of the results and potential directions for future research.

%% file: sections/background.tex
\section{Background} \label{sec:background}

In this section, we introduce the technical foundation necessary for understanding the Continuous Graham Scan algorithm. We begin by stating the assumptions under which our analysis holds and introduce the notation used throughout the paper. Next, we motivate the algorithm by examining the structure of convex envelopes of univariate polynomials, highlighting the role of two-point tangent functions in forming the tightest relaxation. Finally, we present the classical Graham Scan for convex hulls and the Continuous Graham Scan to illustrate the similarity between both algorithms. 

\subsection{Assumptions}\label{assump:bit} 

{
 Throughout the paper, we assume real-number arithmetic is precise up to a relative error of at most $\epsilon$. In other words, except for catastrophic cancellations, we treat approximations within $\epsilon$ relative precision as numerically accurate. Our algorithm depends on previous work where accuracy and complexity are described via ``bit complexity'' $b$; in such a case, we establish the relation via $\epsilon := 2^{-b}$.

We further assume that given a degree-$n$ polynomial $p$, differentiating and evaluating $p$ takes $O(n)$ arithmetic operations. This is because differentiating $p$ requires simple manipulation of a length $n$ sequence of coefficients, and evaluating $p$ using the nested form $p(x) = c_0 + x(c_1 + x(c_2+\cdots))$ evidently requires $O(n)$ additions and multiplications.

Finally, throughout the problem, whenever we apply algorithms to univariate polynomials over some interval $[x_L, x_U]$, we assume that $x_U-x_L \in O(1)$. We remark that polynomials defined over other intervals can be rescaled to fit an assumption.
}

\subsection{Notation}\label{assump:note}
We will use the following notation throughout the paper:
\begin{enumerate}[(1)]
    \item { We use $f$ to denote general univariate functions, and $p$ to denote univariate polynomials. }
    \item { For a function $f$ and a compact interval $\cI \subset \R$, we say $g$ is a \textbf{convex relaxation} of $f$ over $\cI$ if $f(x) \ge g(x)$ for all $x \in \cI$ and $g$ is convex on $\cI$. We say that $g_0$ is the \textbf{convex envelope} (tightest convex relaxation) of $f$ over $\cI$ if $g_0$ is convex and $g_0(x)\ge g(x)$ for all $x\in \cI$ and for any convex relaxation $g$ of $f$ on $\cI$.}
    \item For any univariate function $f$, we use $f|_I$ to denote $f$ being restricted to set $I$. Specifically, it holds that $f|_I(x) = f(x)$ if $x\in I$, and $f|_I(x)=+\infty$ otherwise. 
    \item For any univariate function $f$, we use $f|_I^*$ to denote the convex conjugate of $f|_I$, namely {$f|_I^*(s) := \sup_{x\in\R}\{sx-f|_I(x)\} = \sup_{x\in I}\{sx-f(x)\}$}.
    \item Given $S, T \subset \R$, we use $S \le T$ to denote that $s \le t$ for all $(s, t) \in S \times T$.
\end{enumerate}

    


\subsection{Motivation: Convex Envelopes and Tangents}

Before introducing the algorithm, we first describe the structure of convex envelopes of univariate polynomials, which motivates our approach.  The construction~\eqref{graham.env} of $e_{\cI} p$ and~\Cref{thm:graham} in Section~\ref{sec.graham} formalize the following  remark.

\begin{remark}\label{rk:piecewise-bitan}
Let $\cI = [x_L, x_U]$ be a compact interval and $p: \R \to \R$ be univariate polynomial. Then there exist points $x_L = x_0 < x_1 < \cdots < x_m = x_U$, such that the convex envelope of $p$ can be defined piecewise as 
$$e_{\cI} p(x)  := \sum_{i=0}^{m-1} \mathbbm{1}(x \in [x_i, x_{i+1}]) \cdot g_i(x)$$ 
where exactly one of the following must hold for each $g_i$:
\begin{enumerate}
    \item $g_i = p$ and $g_i$ is convex over $[x_i, x_{i+1}]$
    \item $g_i$ is an affine function. 
\end{enumerate}    
Furthermore, if any of the $g_i$ is affine, then it touches the graph of $p$ at the two points $x_i, x_{i+1}$.

\end{remark}

This remark shows that to compute the tight convex envelope of a univariate polynomial $p$ over a closed interval, it  suffices to identify a collection of affine two-point tangents $\{l_i\}_{i \ge 0}$ connecting the convex intervals of $p$. These tangents must satisfy:
\begin{enumerate}[(i)]
    \item Their slopes increase from left to right, ensuring convexity.
    \item They approximate $p(x)$ as closely as possible, guaranteeing tightness of the envelope.
\end{enumerate}


The above remark indicates that to compute the tight convex envelope for a univariate polynomial $p$ over a closed interval, it suffices to identify all the affine 2-point tangent functions mentioned above, which we call $l_1, l_2,\dots, l_r$. Since these $l_i$'s form part of the tight convex envelope of $p$, intuitively they must (i) have increasing slope if arranged from left to right by the $x$ coordinate, since convex functions (therefore tight convex envelopes) have monotone derivative, and (ii) collectively approximate the graph of $p$ as close as possible, since we are computing $p$'s \textit{tightest} possible convex under-estimator. 

Requirements (i) and (ii) motivate an algorithm to compute tight convex envelopes for a univariate polynomial $p$. At a high level, given a polynomial $p$, we (a) first compute the set of maximal intervals $I_0, I_1, \dots, I_k$, on each of which $p$ is convex, then (b) compute a collection of affine functions $l_1, l_2,\dots, l_r$, each matching $p$ at some points in two of the above intervals, and finally (c) enforce the monotonicity of the slopes of the $l_i$'s by ``removing'' intervals $I_j$'s that violates such monotonicity and then recompute some of the tangent lines, using a mechanism that mimics the famous Graham Scan algorithm of \cite{graham1972efficient} which is used to construct the convex hull of 2D points.




\subsection{Continuous Graham Scan: Algorithm and Illustration}


To provide intuition for our proposed Continuous Graham Scan algorithm (\Cref{alg:graham}), it is helpful to first recall the classical Graham Scan \cite{graham1972efficient}, which is widely used to compute the convex hull of a set of 2D points. The key idea in both algorithms is to maintain a stack of line segments with monotonically increasing slopes, removing elements when convexity is violated. Our continuous version extends this principle from discrete points to univariate functions by constructing affine “bitangent” segments that together form the convex envelope of a polynomial.  

The following presentation compares the classical Graham Scan and our continuous variant, highlighting their structural similarity. We first show the classical algorithm in \Cref{alg:og-graham}, followed by our Continuous Graham Scan in \Cref{alg:simp-graham}. To illustrate the mechanisms, we provide step-by-step visualizations for both algorithms applied to a representative polynomial. These figures emphasize how monotonicity in slopes ensures convexity and how problematic segments are skipped or removed, establishing the tight convex envelope in the continuous setting.

\begin{algorithm}[H]
    \caption{Original Graham Scan (Modified from \cite{graham1972efficient}; Alternative Presentation)}
    \begin{algorithmic}\label{alg:og-graham}
        \State{\textbf{Input}: points $(x_0, y_0), (x_1, y_1),\dots, (x_k, y_k)$ sorted by the $x$ coordinate}
        \State{Initialize empty stack $S$}
        \For{$i=0,1,2,...,k-1$}
            \State{$l \gets $ the line through $(x_i,y_i)$ and $(x_{i+1}, y_{i+1})$}
            \While{$S$ not empty and $\texttt{top}(S)$ has slope greater than that of $l$}
            \State{$l_{bad} := S.\texttt{pop}()$}
            \State{$(x_j,y_j) \gets$ the left vertex of $l_{bad}$}
            \State{$l \gets $ the line through $(x_j,y_j)$ and $(x_{i+1}, y_{i+1})$}
            \EndWhile
            \State{$S.\texttt{push(}l\texttt{)}$}
        \EndFor
        \State{\Return{$S$}}
    \end{algorithmic}
\end{algorithm}
\vspace{-1mm}
\begin{algorithm}[H]
    \caption{Continuous Graham Scan (Our Main Algorithm; Alternative Presentation)}
    \begin{algorithmic}\label{alg:simp-graham}
        \State{\textbf{Input}: univariate polynomial $p$ of degree $n$, compact interval $\cI = [x_L,x_U] \subset \R$}
        \State{Initialize empty stack $S$}
        \State{compute the maximal intervals $I_0, I_1, \dots , I_k$ of $\cI$ on which $p$ is convex}
        \For{$i=0,1,...,k-1$}
            \State{$l \gets \bitan_p(I_i, I_{i+1})$}
            \While{$S$ not empty and $\texttt{top}(S)$ has slope greater than that of $l$}
            \State{$l_{bad} := S.\texttt{pop}()$}
            \State{$I_j \gets$ the left convex interval of $l_{bad}$}
            \State{$l \gets \bitan_p(I_j, I_{i+1})$}
            \EndWhile
            \State{$S.\texttt{push(}l\texttt{)}$}
        \EndFor
        \State{\Return{$S$}}
    \end{algorithmic}
\end{algorithm}



\Cref{alg:og-graham} computes the lower half of the convex hull for a set of 2D points by maintaining a stack of line segments and enforcing monotone slopes, removing any segment that would violate convexity. Our Continuous Graham Scan (\Cref{alg:simp-graham}) generalizes this principle to univariate polynomials. Rather than connecting discrete points, it constructs affine “bitangent” segments that are tangent to the polynomial at two points and lie below the function. The algorithm preserves the monotone slope property, discarding segments that would break convexity. The two main operations e.g., identifying convex intervals and computing the bitangent segments, are detailed in the following sections.

To build intuition, the following figures illustrate both the discrete and continuous algorithms on a representative polynomial  
$$p(x) := 1.5x + 1.3x^2 -0.7x^4 + 0.08x^6 -0.0025x^8$$ 
over the interval $[-4.1, 4.4]$. In both cases, the stack of line segments enforces convexity by maintaining monotonically increasing slopes. In the discrete Graham Scan, this is done over individual points, while in the continuous version it is done over convex intervals and their associated bitangent segments.  

The figures highlight the operational analogy between the two algorithms. Orange segments indicate convex intervals of the polynomial, green segments are the selected bitangents that satisfy the monotone slope condition, and red segments represent segments that violate this condition and are therefore removed. Here, the figures show how the continuous algorithm systematically constructs the tight convex envelope, mirroring the logic of the classical Graham Scan while generalizing it to univariate functions. These illustrations provide a concrete, intuitive understanding of the mechanism before we formally define convex intervals and bitangent computation in the next section.

\begin{center}
    \begin{figure*}
        \subfloat[Connecting points 0 and 1.]{\vstretch{.9}{\includegraphics[width=.44\textwidth]{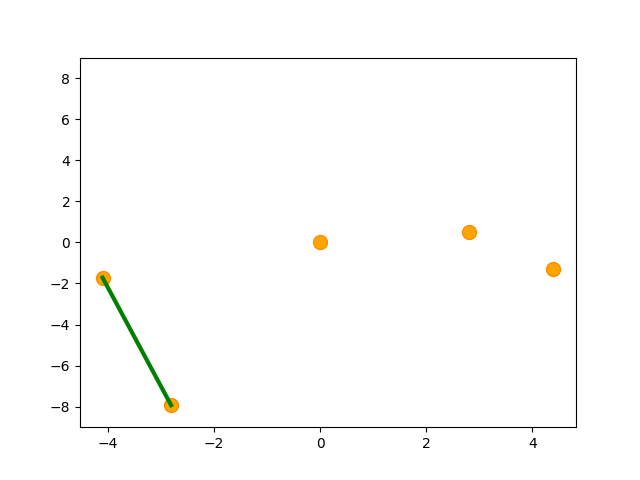}}}\hfill
        \subfloat[Connecting points 1 and 2; the slopes of line segments increases from left to right.]{\vstretch{.9}{\includegraphics[width=.44\textwidth]{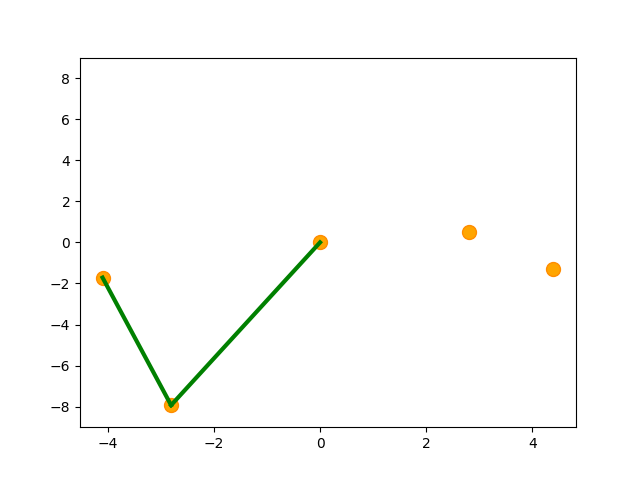}}}\\
        \subfloat[Connecting points 2 and 3 is problematic, as the sequence of line segments is no longer monotone in slope.]{\vstretch{.9}{\includegraphics[width=.44\textwidth]{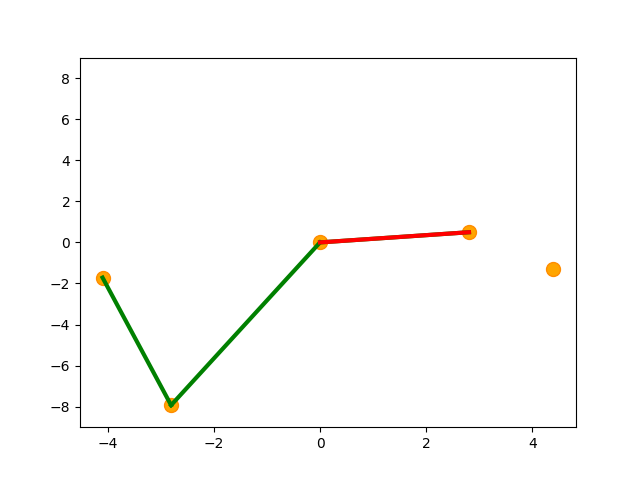}}}\hfill
        \subfloat[We then ``skip'' point 2 and move on, because point 2 will never be part of the convex hull of the points.]{\vstretch{.9}{\includegraphics[width=.44\textwidth]{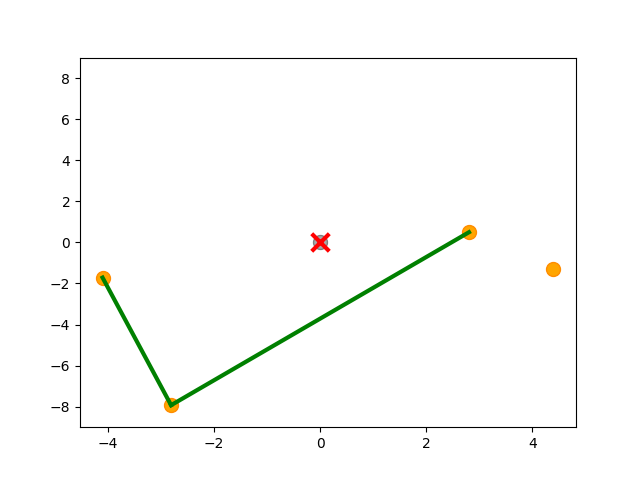}}}\\
        \subfloat[Connecting points 3 and 4 is once again problematic, as the sequence of line segments is no longer monotone in slope.]{\vstretch{.9}{\includegraphics[width=.44\textwidth]{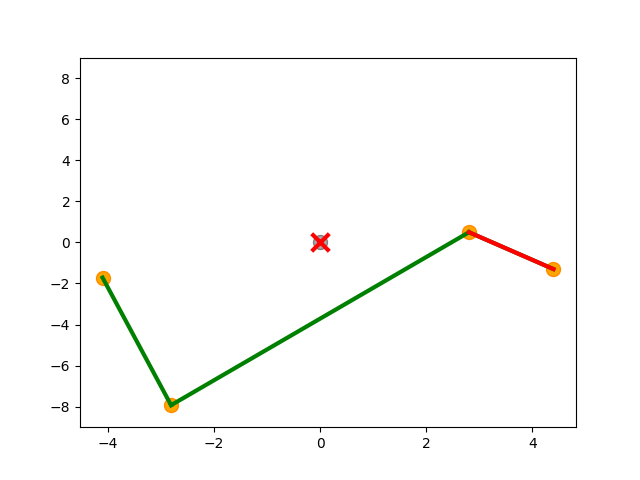}}}\hfill
        \subfloat[We ``skip'' point 3 and move on, because point 3 will never be part of the convex hull of the points]{\vstretch{.9}{\includegraphics[width=.44\textwidth]{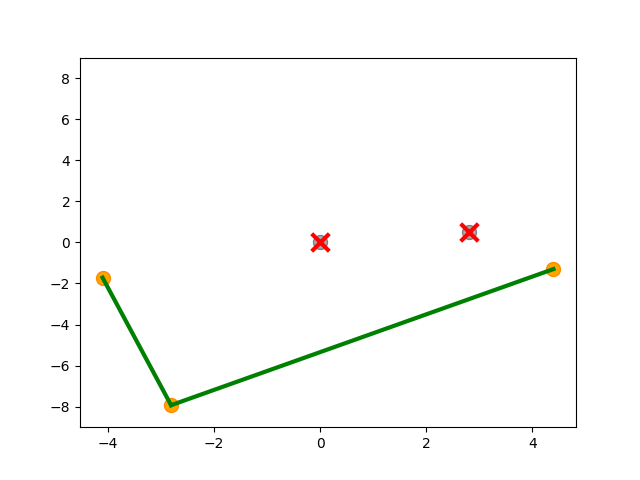}}}\\\\
        \caption{Classical Graham Scan Algorithm (\Cref{alg:og-graham}) for the discrete convex hull. Although originally designed for computing the convex hull of 2D points, here we only show how it computes the ``bottom half'' of the convex hull, to emphasize its relation to our continuous approach for function convex envelopes.}
    \label{fig:placeholder}
    \end{figure*}
\end{center}

\begin{center}
    \begin{figure*}
        \subfloat[Constructing a bitangent between the left endpoint and the first convex interval]{\vstretch{1}{\includegraphics[width=.44\textwidth]{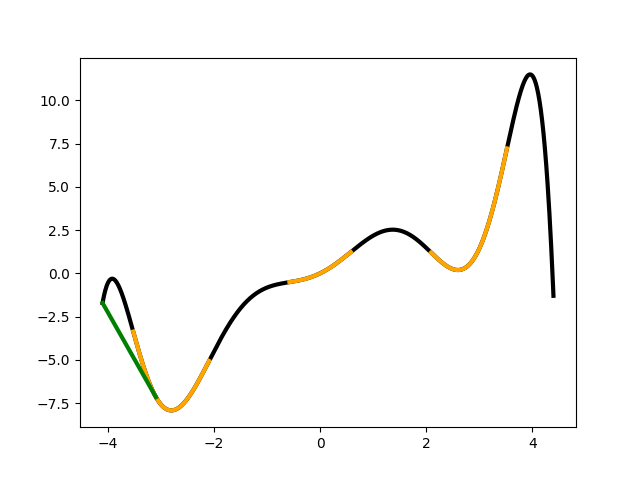}}}\hfill
        \subfloat[Constructing a bitangent between the next two convex intervals]{\vstretch{1}{\includegraphics[width=.44\textwidth]{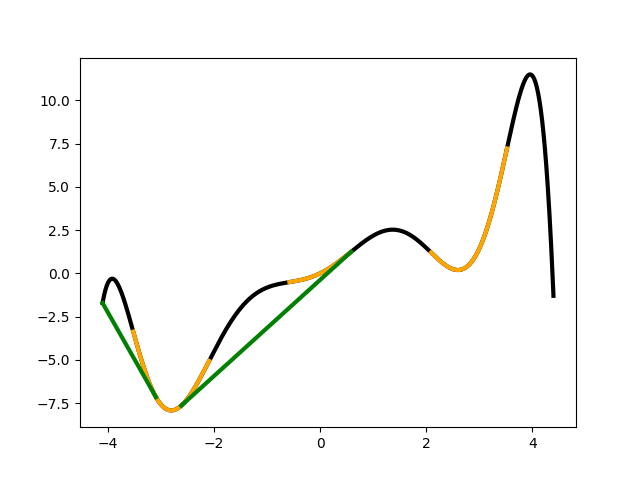}}}\\
        \subfloat[Constructing the next bitangent; notice that the red bitangent is problematic, as the sequence of bitangents is no longer monotone]{\vstretch{1}{\includegraphics[width=.44\textwidth]{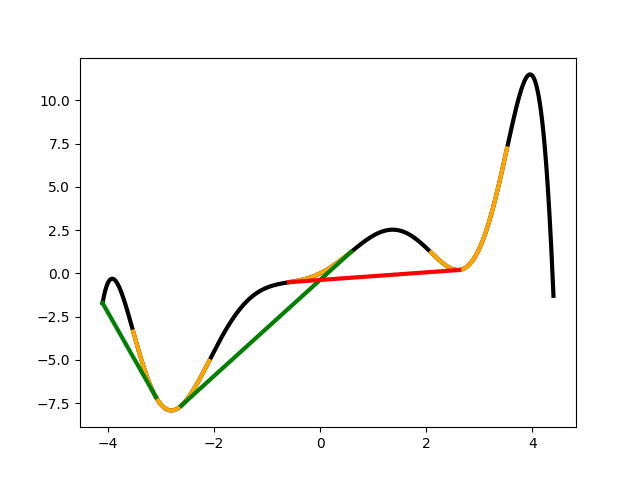}}}\hfill
        \subfloat[We ``skip'' the convex interval between the faulty bitangents and move on, as it will never be part of the convex envelope of $p(x)$]{\vstretch{1}{\includegraphics[width=.44\textwidth]{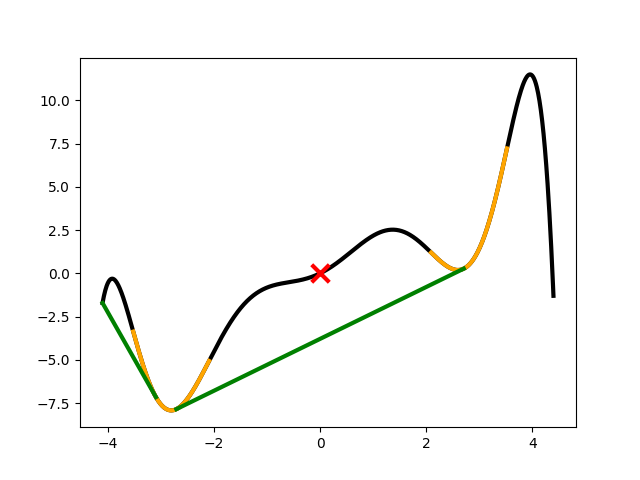}}}\\
        \subfloat[Constructing the next bitangent is once again problematic, as the sequence of bitangents is no longer monotone in slope]{\vstretch{1}{\includegraphics[width=.44\textwidth]{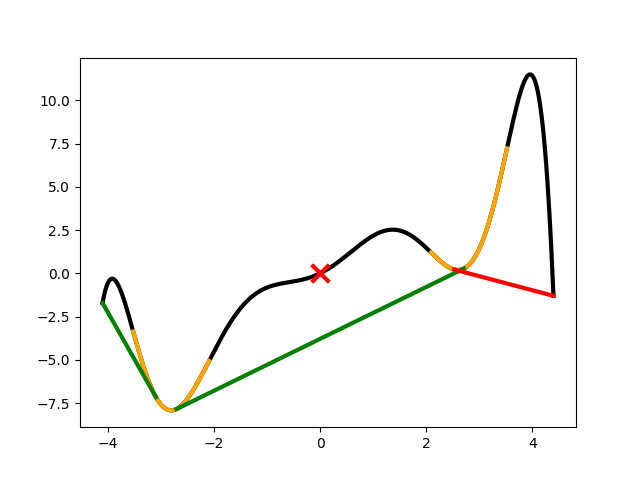}}}\hfill
        \subfloat[We constructed another problematic bitangent (red), as the bitangents are not monotone. ]{\vstretch{1}{\includegraphics[width=.44\textwidth]{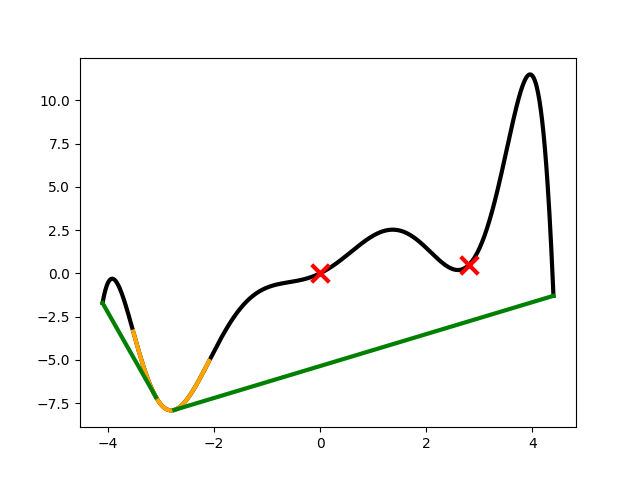}}}\\\\
        \caption{Our Continuous Graham Scan Algorithm (\Cref{alg:simp-graham}) for the convex envelope of univariate functions. Notice that its mechanism is quite comparable to that of the discrete graham scan algorithm.}
    \end{figure*}
\end{center}

%% file: sections/roots.tex
\section{Continuous Graham Scan for Univariate Convex Envelope}\label{sect:graham}

Following the idea from the motivations section, we present our algorithm for computing the convex envelope of a univariate polynomial $p$ over some domain $[x_L, x_R]$. 
We break the problem into three components, which allows us to draw a clear analogy with the classical Graham Scan, making it easier to adapt its geometric intuition to the continuous setting.

\vspace{-2mm}
\begin{enumerate}[(i)]
    \item Identify the maximal intervals $I_0,I_1, \dots, I_k$ in which $p$ is convex; these correspond to the ``2D points'' in the vanilla Graham scan.
    \item Design a subroutine that, given some $I_i, I_j$ above, computes the unique affine function that underestimates $p$ on $I_i \cup I_j$ and intersects the graph of $p$ at exactly one point in each of $I_i, I_j$. This is analogous to connecting two ``2D points'' together.
    \item Modify the vanilla Graham Scan algorithm using (i) and (ii).
    \vspace{-2mm}
\end{enumerate}
We present each of the components in the following three subsections.

\subsection{Computing Convex Intervals of a Univariate Polynomial}\label{sec:ci}


As the first step, we identify intervals where the polynomial is convex, which will play the role of “points” in the discrete Graham Scan.  This is formalized in the following  lemma.

\begin{lemma}\label{lemma.roots}
Fix any univariate polynomial $p: \R \to \R$ of degree $n$ and any compact interval $\cI :=[x_L,x_U]\subset \R$. Then the set 
\[
[x_L,x_U] \setminus \{x \in (x_L,x_U): p''(x) < 0\}
\]
is a finite union of $O(n)$ mutually disjoint closed intervals.  
\end{lemma}
\begin{proof}
Since $p''$ is a polynomial of degree $n-2$, it has at most $n-2$ real roots in 
$(x_L,x_U)$. Let $r_1<\cdots<r_m$ be the (possibly empty) ordered set of real roots of $p''$ in $(x_L,x_U)$ and let $r_0:=x_L$ and $r_{m+1}:=x_U$. Then
\[
\{x \in (x_L,x_U): p''(x) \ne 0\} = (r_0,r_{m+1}) \setminus \{r_1,\dots,r_{m}\} = \bigcup_{j=0}^m (r_j,r_{j+1}).
\]
The latter set is evidently a union of finitely many disjoint open subintervals of  $(x_L,x_U) = (r_0,r_{m+1})$.  Since the sign of $p''$ cannot change in any of the open intervals $(r_j,r_{j+1})$, each of the following two open sets is also a union of finitely many disjoint open intervals of the form $(r_j,r_{j+1})$ for some $j=0,\dots,m$: 
\[
\{x \in (x_L,x_U): p''(x) < 0\} \text{ and } \{x \in (x_L,x_U): p''(x) > 0\}.
\]
In particular,
\begin{equation}\label{eq.concave}
\{x \in (x_L,x_U): p''(x) < 0\} = \bigcup_{i=1}^k (r_{j_i},r_{j_i+1})
\end{equation}
for some $j_1,j_2,\dots,j_k \in \{0,1,\dots,m\}$. By convention, we take $k=0$ when 
$\{x \in (x_L,x_U) \ | \ p''(x) < 0\} = \emptyset$ so that the union in the above right-hand-side is empty.  Therefore by taking $j_0 = -1$ and $j_{k+1} = m+1$ we get
\begin{equation}\label{eq.convex}
[x_L,x_U] \setminus \{x \in (x_L,x_U): p''(x) < 0\} = [r_0,r_{m+1}]\setminus
\bigcup_{i=1}^k (r_{j_i},r_{j_i+1}) = 
\bigcup_{i=0}^k [r_{j_i+1},r_{j_{i+1}}].
\end{equation}
Observe that the latter set is $[x_L,x_U] = [r_0,r_{m+1}]$, as it should be, when $\{x \in (x_L,x_U) \ | \ p''(x) < 0\} = \emptyset$ as in that case $k=0$.
\end{proof}

Now we formalize the concept of convex intervals.

\begin{defn}\label{def:ci}
    For any degree $n$ univariate polynomial $p$ and a compact interval $\cI := [x_L, x_U] \subset \R$, we define the set of \textbf{Convex Intervals} of $p$ with respect to $\cI$, which is denoted as $CI_\cI(p)$, to be the set of mutually disjoint closed intervals $\{I_0, I_1, \dots, I_k\subset \cI\}$, where $I_i \le I_{i+1}$ for $i=0,\dots,k-1$ and $
[x_L,x_U] \setminus \{x \in (x_L,x_U): p''(x) < 0\} = I_0\cup I_1 \cup\cdots \cup I_k$ whose existence is guaranteed by Lemma~\ref{lemma.roots}.
\end{defn}

    

The set $CI_\cI(p)=\{I_0, I_1, \dots, I_k\}$ contains all the maximal closed subintervals of $[x_L,x_U]$ where $p$ is convex.  It will always holds that $x_L \in I_0$ and $x_U \in I_k$ which will be key to ensure the correctness of our Continuous Graham Scan algorithm.  We note that it is possible to have $I_0 = [x_L,x_L] = \{x_L\}$ and also $I_k=[x_U,x_U] = \{x_U\}$.

The computation of convex intervals of a univariate polynomial $p$ requires finding all { real roots of $\frac{d^2}{dx^2}p(x)$, a polynomial of degree $n-2$}, which can be approximated to a { relative error of $\epsilon$ (ie. our assumed machine precision in \Cref{assump:bit}) in $O(n^3 + n \log^2 n \log b) = O(n^3 + n \log^2 n \log \log \epsilon^{-1})$ time} by solving the eigen-problem for the companion matrix of $p$, according to  \citet{pan1999complexity}. With the real roots of a polynomial $p$ already computed, we compute the convex intervals of $p$ as shown in \Cref{alg:ci}. 

\begin{algorithm}[H]
    \caption{Convex Interval Computation}
    \begin{algorithmic}\label{alg:ci}
        \State{\textbf{Input}: univariate polynomial $p$ of degree $n$, closed interval $[x_L, x_U] \subset \R$}, {real roots of $p''$}
        \State{$R \gets \{ \text{real roots of } p'' \text{ within }(x_L, x_U)\} \cup \{x_L, x_U\}$}
        \State{let $x_L = r_0 < r_1 < \cdots < r_{m+1} = x_U$ be the elements of $R$ in ascending order}
        \State{$k \gets 0$ and $j_k\gets -1$}
        \For{$j=0,1,2,\dots,m$}
        \If{$p''(\frac{r_j + r_{j+1}}{2}) < 0$}
        \State{$k\gets k+1$ and $j_k\gets j$}   
        \EndIf
        \EndFor
        \State \Return{$C=\{[r_0,r_{j_1}],[r_{j_1+1},r_{j_2}],\dots,[r_{j_k+1},r_{m+1}]\}$}
    \end{algorithmic}
\end{algorithm}

The following theorem establishes both the correctness and efficiency of \Cref{alg:ci}, showing that it reliably identifies all convex intervals in polynomial time. 

\begin{theorem}\label{thm:ci}
    {Let $p$ be a univariate polynomial of degree $n$, let $\cI \subset \R$ be a closed interval. When given the real roots of $p$, \Cref{alg:ci} terminates in $O(n^2)$ arithmetic operations and returns $CI_\cI(p)$. }
\end{theorem}
\begin{proof}
\Cref{alg:ci}'s runtime consists of two parts, sorting the $m \le n$ real roots, which takes $O(n\log n)$, and computing $j_1,\dots,j_k$, which requires $m+1 \le n$ iterations of the for loop. Within each iteration, there is one evaluation of $p''$ plus a few constant-time steps. Since evaluation of $p''$ takes $O(n)$ according to our assumption in \Cref{assump:bit}, the overall runtime is $O(n^2)$.  
The construction of the for loop guarantees that  $j_1,\dots,j_k$ are such that~\eqref{eq.concave} holds.  Thus $C$ returns the intervals in~\eqref{eq.convex}, which is precisely $CI_\cI(p)$ as per \Cref{def:ci}.
\end{proof}

%% file: sections/bitan_no_mark.tex
\subsection{Computation of Unique Bitangents for Polynomials}\label{pf:bitan}

In this section, we present and prove the correctness of \Cref{alg:bisect}, which is an algorithm that, given a polynomial $p$, and two distinct convex intervals $I_1, I_2$ in some $CI_\cI(p)$ (defined in \Cref{def:ci}), computes a \textit{bitangent} of $p$ through $I_1, I_2$. 
Informally, a bitangent is an affine function that lies below $p$ and touches each of the intervals $I_1$ and $I_2$. A formal definition is given below.

\begin{defn}\label{def:bitan}
   Let  $p$ be a univariate polynomial, and  $I_1, I_2 \subseteq \R$ be disjoint closed intervals.  The  affine function $\bitan_p(I_1,I_2)$ is the \textbf{bitangent} associated with $p, I_1, I_2$ if $\bitan_p(I_1, I_2)(x) \le p(x)$ for all $x \in I_1 \cup I_2$, and $p-\bitan_p(I_1, I_2)$ has at least one root in each of $I_1$ and $I_2$.
\end{defn}
\Cref{lemma:binary} below implies that $\bitan_p(I_1, I_2)$  always exists and is unique. 

In a relevant work from INRIA \cite{katsamaki2023exact}, they compute the exact convex hull of rational parametric curves via finding bitangent lines by solving polynomial equations. 
This observation highlights that, unlike general methods which may need to solve high-degree polynomial equations with multiple roots, our problem structure guarantees a single solution per convex interval pair, making the computation more tractable.
However, since our algorithm mimics that of Graham Scan in \cite{graham1972efficient}, we only need to compute bitangents with respect to specific pairs of convex intervals. As per \Cref{cor.bitangent}, each such pair of convex intervals admit a unique bitangent; therefore intuitively, identifying such unique solution is easier than generally solving for polynomial equations that may have multiple roots.

{ The following lemma shows that given two disjoint closed intervals $I_1, I_2$ and a univariate polynomial $p(x)$, the bitangent $\bitan_p(I_1,I_2)$ can be efficiently computed by solving the equation $p|_{I_1}^*(s)=p|_{I_2}^*(s)$ via binary search as detailed in
\Cref{alg:bisect}.  See also \Cref{fig:bitangent-conju} for an illustration.

\begin{lemma}\label{lemma:binary} Let  $p$ be a degree $n$ univariate polynomial, and  $I_1, I_2 \subseteq \R$ be disjoint closed intervals such that $I_1 \le I_2$.
\begin{itemize}
\item[(a)] Let $ I := \conv(I_1\cup I_2)$, and suppose  $m \le \min_{x\in I} p'(x)$ and $M \ge \max_{x\in I} p'(x)$.  Then $p|_{I_1}^*(m) \ge p|_{I_2}^*(m)$ and $p|_{I_1}^*(M) \le p|_{I_2}^*(M)$.  In particular, $p|_{I_1}^*(s_0)=p|_{I_2}^*(s_0)$ for some $s_0 \in [m,M].$ 
\item[(b)] Suppose $s_0\in\R$ is such that $p|_{I_1}^*(s_0)=p|_{I_2}^*(s_0)$.  
Then the affine function
\[
g(x):= s_0x - p_{I_1}^*(s_0) =s_0x - p_{I_2}^*(s_0)
\]
satisfies the conditions of Definition~\ref{def:bitan}, that is, $g(x) \le p(x)$ for all $x\in I_1\cup I_2$  and $p-g$ has at least one root in each of $I_1$ and $I_2$.

\item[(c)] Suppose $s_0\in\R$ is such that $p|_{I_1}^*(s_0)=p|_{I_2}^*(s_0)$. Then 
$p|_{I_1}^*(s)<p|_{I_2}^*(s)$ for all $s > s_0$ and $p|_{I_2}^*(s)< p|_{I_1}^*(s)$ for all $s < s_0$. 

\end{itemize}
\end{lemma}
\begin{proof}

\begin{itemize}
\item[(a)] Since $I_1 \le I_2$, the choice of $m,M$ and the mean value theorem imply that for all $x\in I_1$ and $y\in I_2$
\[
m \le \frac{p(y)-p(x)}{y-x} \le M,
\]
or equivalently
\[
mx - p(x) \ge my - p(y) \text{ and } Mx - p(x) \le My - p(y).
\]
Since the above two inequalities hold for all $x\in I_1$ and $y\in I_2$ it follows that
\[
p_{I_1}^*(m) = \sup\{mx - p(x): x\in I_1\} \ge \sup\{my - p(y): y\in I_2\} = p_{I_2}^*(m)
\]
and
\[
p_{I_1}^*(M) = \sup\{Mx - p(x): x\in I_1\} \ge \sup\{My - p(y): y\in I_2\} = p_{I_2}^*(M).
\]
Furthermore, since both $p_{I_1}^*$ and $p_{I_2}^*$ are continuous functions, the intermediate value theorem implies that $p|_{I_1}^*(s_0)=p|_{I_2}^*(s_0)$ for some $s_0 \in [m,M].$
\item[(b)] The construction of $g$ readily implies that $g(x) = s_0 x - p_{I_1}^*(s_0) \le s_0 x - (s_0 x - p(x)) = p(x)$ for all $x\in I_1$ and 
$g(x) = s_0 x - p_{I_2}^*(s_0) \le s_0 x - (s_0 x - p(x)) = p(x)$ for all $x\in I_2$.
In other words, $g(x) \le p(x)$  for all $x\in I_1\cup I_2$.  Furthermore, for $x_1:=\argmin\{s_0x-p(x): x\in I_1\}$ 
\[
g(x_1) = s_0 x_1 - p_{I_1}^*(s_0) = p(x_1).
\]
Similarly, $g(x_2) = p(x_2)$ for $x_2:=\argmin\{s_0x-p(x): x\in I_2\}$.  Thus $p-g$ has at least one root in each of $I_1$ and $I_2$.

\item[(c)] Let $x_1\in I_1, x_2\in I_2$ and $g$ be as in part (b).

First, consider the case $s>s_0$.  Let $z_1:=\argmax\{sx - p(x): x\in I_1\}$. Since $g$ is affine with slope $s_0$, $z_1\in I_1, x_2\in I_2$, and $I_1\le I_2$ we have
\[
p(x_2) = g(x_2) = g(z_1) + s_0(x_2-z_1) < p(z_1) + s(x_2 - z_1).
\]
Thus
\[
p_{I_1}^*(s) = sz_1 - p(z_1) < sx_2 - p(x_2) \le \sup\{sx-p(x): x\in I_2 \} = p_{I_2}^*(s).
\]

Second, consider the case $s<s_0$.  Let $z_2:=\argmax\{sx - p(x): x\in I_2\}$. Since $g$ is affine with slope $s_0$, $x_1\in I_1, z_2\in I_2$, and $I_1\le I_2$ we have
\[
p(x_1) = g(x_1) = g(z_2) + s_0(x_1-z_2) < p(z_2) + s(x_1 - z_2).
\]
Thus
\[
p_{I_2}^*(s) = sz_2 - p(z_2) < sx_1 - p(x_1) \le \sup\{sx-p(x): x\in I_1 \} = p_{I_1}^*(s).
\]

\end{itemize}
\end{proof}

\begin{corollary}\label{cor.bitangent}
Suppose $p$ is a univariate polynomial, $I_1, I_2\subseteq \R$ are disjoint compact intervals such that $I_1\le I_2$ and let $s_0$ the the unique solution to $p_{I_1}^*(s) = p_{I_2}^*(s)$.  Then $\bitan_p(I_1, I_2)$ is the following uniquely defined affine function:
\[
\bitan_p(I_1, I_2)(x):=p(x) + p_{I_1}^*(s_0)=p(x) + p_{I_2}^*(s_0).
\] 
Furthermore, if $p$ is concave on $\conv(I_1\cup I_2) \setminus (I_1\cup I_2)$ then $\bitan_p(I_1,I_2) \le p$ on 
$\conv(I_1\cup I_2)$. 
\end{corollary}
\begin{proof}
The uniqueness of $s_0$ and of $\bitan_p(I_1, I_2)$ readily follows from~\Cref{lemma:binary}.  Suppose $p$ is concave on $\conv(I_1\cup I_2) \setminus (I_1\cup I_2)$.  Observe that 
\[\conv(I_1\cup I_2) \setminus (I_1\cup I_2) = (u_1, u_2)\]
for $u_1 = \max I_1$ and $u_2 = \min I_2$.
To show the second statement it suffices to show that $\bitan_p(I_1, I_2)(x)\le p(x)$ for $x\in (u_1, u_2)$ since $\bitan_p(I_1, I_2)\le p$ on $I_1\cup I_2$.  Indeed, for
$x\in (u_1, u_2)$ we have 
\[
x = \frac{x-u_1}{u_2-u_1} \cdot u_2 + \frac{u_2-x}{u_2-u_1} \cdot u_1.
\] 
Hence the concavity of $p$ on $(u_1, u_2)$ and the fact that $\bitan_p(I_1, I_2)$ is affine and $\bitan_p(I_1, I_2)\le p$ on $I_1\cup I_2$ imply that
\[
\begin{aligned}
\bitan_p(I_1, I_2)(x) &= \frac{x-u_1}{u_2-u_1} \cdot 
\bitan_p(I_1, I_2)(u_2) + \frac{u_2-x}{u_2-u_1} \cdot \bitan_p(I_1, I_2)(u_1) 
\\
&\le \frac{x-u_1}{u_2-u_1} \cdot 
p(u_2) + \frac{u_2-x}{u_2-u_1} \cdot p(u_1) 
\\&\le p(x).
\end{aligned}
\]
\end{proof}

\Cref{lemma:binary} also suggests a natural binary procedure to find $s_0$ as detailed in \Cref{alg:bisect}.

We emphasize that efficiently evaluating $p|_I^*$, the convex conjugate of $p$ restricted to $I$ (defined in \Cref{assump:note}), is important for computing the bitangent $\bitan_p(I_1, I_2)$, because the sign of $p|_{I_1}^*(s)-p|_{I_2}^*(s)$ can be used to binary-search for the slope of $\bitan_p(I_1, I_2)$. 
As formalized in \Cref{lem:conj}, evaluating $p|_I^*(s)$ can be done efficiently.
The correctness of such binary-search approach is formally shown in \Cref{thm:bitangent}.

\begin{figure}
\begin{center}
    \includegraphics[width=16cm]{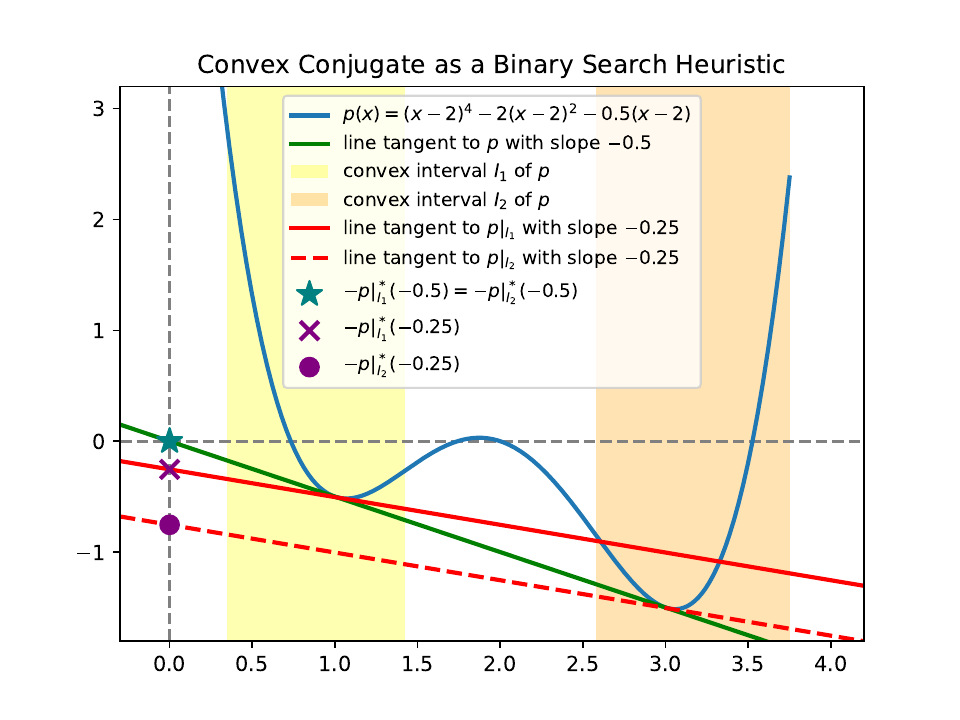}
    \caption{Illustration of using Convex Conjugate to binary search for the correct slope of bitangents: we use $p(x) := (x-2)^4 - 2(x-2)^2 - 0.5(x-2)$ restricted to the compact interval $[0.25, 3.75]$. In this case, its Convex Intervals are approximately $\{I_1, I_2\} := \{[0.25, 1.423], [2.577, 3.75]\}$ and $p|_{I_1}^*(s_0) = p|_{I_2}^*(s_0) = 0.0$ for $s_0=-0.5$. Thus $\bitan_p(I_1, I_2)(x) = -0.5x$. On the other hand, for  $s = -0.25 > -0.5 = s_0$ we have $ -p|_{I_1}^*(s) \approx -0.254 >  -0.754\approx -p|_{I_2}^*(s).$}
    \label{fig:bitangent-conju}
\end{center}
\end{figure}
}

\begin{algorithm}[H]
    \caption{Bisection Algorithm for Computing Bitangents }
    \begin{algorithmic}\label{alg:bisect}
        \State{\textbf{Input}: univariate polynomial $p$ of degree $n$, convex intervals $I_l, I_r$ of $p$}
        \State{Find $m <0$ and $ M > 0$ (via a doubling scheme) such that $p_{I_l}^*(m) \ge p_{I_r}^*(m)$ and $p_{I_l}^*(M) \le p_{I_r}^*(M)$ }
        \State{$s_{\min}\gets m$ and $s_{\max}\gets M$}
        \For{$t = 1,\dots, T = \ceil{1+\log((M-m)/\epsilon)}$}
        \State{$s_{\text{mid}} \gets (s_{\min}+s_{\max})/2$}
        \If{$p|_{I_l}^*(s_{\text{mid}}) \le p|_{I_r}^*(s_{\text{mid}})$}
            \State{$s_{\max} \gets s_{\text{mid}}$}
        \Else{}
            \State{$s_{\min} \gets s_{\text{mid}}$}
        \EndIf
        \EndFor
        \State{$s_0 \gets \frac{s_{\min}+s_{\max}}{2}$}
        \State \Return{$s_0x + p|_{I_l}^*(s_0)$}
    \end{algorithmic}
\end{algorithm}

{
The lemma and theorem validate \Cref{alg:bisect}.\\
}

\begin{lemma}[Efficient Evaluation of $p|_I^*(s)$]\label{lem:conj}
    For any univariate polynomial $p$ of degree $n$ and { some $O(1)$-width} compact interval $I$ over which $p$ is convex, we denote $p|_I$ as its restriction over $I$ and $p|_I^*$ as the convex conjugate over its restriction, as per \Cref{assump:note}. We can numerically evaluate $p|_I^*(s)$ for any $k\in\R$ up to a { machine precision of $\epsilon$ in $O(n\log(1/\epsilon))$ time.} 
\end{lemma}
\begin{proof}
    { Since $p|_I^*(s) = \sup_{x \in I}\{sx-p(x)\}$, the interval $I$ is compact, and $sx-p(x)$ is continuous, the supremum is attained. It therefore suffices to first compute $x_0 := \argmax_{x\in I} \{sx-p(x)\}$, which then yields $p|_I^*(s) = sx_0-p(x_0)$.}
    
    Let $x_L, x_U$ be the lower and upper endpoints of $I$; since $p$ is convex over $I$, its derivative must be monotone, so we have $\frac{d}{dx}p(x_L) \le \frac{d}{dx}p(x) \le \frac{d}{dx}p(x_U)$ for all $x \in I$. Then we have two cases: 
    \vspace{-5mm}
    \begin{enumerate}[(i)]
        \item If either $s \le \frac{d}{dx}p(x_L)$ or $s \ge \frac{d}{dx}p(x_U)$, then $x_0$ must be respectively either $x_L$ or $x_U$, so they can be identified in $O(1)$ time.
        \item When $\frac{d}{dx}p(x_L) < s < \frac{d}{dx}p(x_U)$. Since $p$ is convex over $I$, its derivative is monotone, so we can binary search for some $x_0\in\left(p(x_L), p(x_U)\right)$ where $\frac{d}{dx}p(x_0) = s$. {With a machine precision of $\epsilon$, this takes $\log_2\left(\frac{x_U-x_L}{\epsilon}\right) = O(\log(1/\epsilon))$ evaluations of $\frac{d}{dx}p$, plus an additional $O(\log(1/\epsilon))$ time. Since evaluation of $\frac{d}{dx}p$ is in $O(n)$, the procedure takes $O(n\log(1/\epsilon))$ overall.}
        \vspace{-5mm}
    \end{enumerate}
    and the casework concludes our proof.
\end{proof}

\begin{theorem}[Correctness of \Cref{alg:bisect}]\label{thm:bitangent}
    Given a univariate polynomial $p$, convex intervals $I_l, I_r$ on $p$, arithmetic precision $ \epsilon$, and a constant $\mathcal L \ge 1$ such that $|p'| \le \mathcal L$  on the interval $\conv(I_l \cup  I_r)$, \Cref{alg:bisect} computes $\bitan_p(I_l, I_r)$ up to arithmetic {precision $\epsilon$ in $O\left(n\log(1/\epsilon)\log (\mathcal L/\epsilon)\right)$ time.}
\end{theorem}
\begin{proof}
Since  $|p'| \le \mathcal L$  on the interval $\conv(I_l \cup  I_r)$, then \Cref{lemma:binary} implies that after at most $\log (2\mathcal L)$ doubling steps the initial step of      
\Cref{alg:bisect} finds $m<0$ and $M>0$ such that  $p_{I_l}^*(m) \ge p_{I_r}^*(m)$ and $p_{I_l}^*(M) \le p_{I_r}^*(M)$ with $|m|\le 2\mathcal L$ and $M \le 2\mathcal L$.   \Cref{lemma:binary} also implies that the slope $s_0$ of $\bitan_p(I_1, I_2)$ must be in the interval $[-2\mathcal L, 2\mathcal L]$. Next, again by \Cref{lemma:binary}, we have $p|_{I_l}^*(s_{\text{mid}}) \le p|_{I_r}^*(s_{\text{mid}})$ if and only if $s_0 \le s_{\text{mid}}$, where $s_0$ is the slope of $\bitan_p(I_l,I_r)$. This means each time we halve the interval $[s_{\min}, s_{\max}]$, the new interval still contains $s_0$. After $T = \ceil{1+\log(4\mathcal L/\epsilon)}$ bisections, $s_0$ is then sandwiched in the interval $[s_{\min}, s_{\max}]$ of size $4\mathcal L \cdot 2^{-1-\log(4\mathcal L/\epsilon)} \le \epsilon$, which is within arithmetic error, and since $s_0$ is correctly identified, the corresponding output must be the unique $\bitan_p(I_l,I_r)$.

    For runtime analysis, the algorithm takes $O(\log(\mathcal L/\epsilon))$ doubling and bisection rounds, and the runtime of each round is dominated by the two calls to the routine in \Cref{lem:conj}, each taking {$O(n\log(1/\epsilon))$ time. The overall runtime is therefore $O\left(n\log(1/\epsilon)\log (\mathcal L/\epsilon)\right)$.}
\end{proof}

%% file: sections/graham.tex
\subsection{Continuous Graham Scan}
\label{sec.graham}
We next present \Cref{alg:graham}, which computes the essential ingredient for the convex envelope of a polynomial $p$ over an interval $\cI$ as we detail below in~\eqref{graham.env}. \Cref{alg:graham} extends the Graham Scan \cite{graham1972efficient} to continuous polynomial graphs, preserving the key geometric intuition while replacing discrete line segments with continuous bitangent segments. Specifically, instead of connecting pairs of 2D points, \Cref{alg:graham} uses
 \Cref{alg:ci} to compute the convex intervals $CI_\cI(p)$,
and 
 \Cref{alg:bisect} to compute the  bitangents between pairs of intervals in $CI_\cI(p)$.   \Cref{thm:graham} below formalizes the correctness of the algorithm and its running time.

\begin{algorithm}[H]
    \caption{Continuous Graham Scan}
    \begin{algorithmic}\label{alg:graham}
        \State{\textbf{Input}: univariate polynomial $p$ of degree $n$, closed interval $\cI = [x_L,x_U]$}
        \State{{Compute $CI_\cI(p) = \{I_0, I_1, \dots, I_k\}$} }
        \State{Initialize empty stack $S$ for the bitangents that we will keep}
        \For{$i=0, 1, \dots, k-1$}
            \State{$l \gets \bitan_p(I_i, I_{i+1})$}
            \While{$S$ not empty and $\texttt{top}(S)$ has slope greater than or equal that of $l$}
            \State{$l_{bad} := S.\texttt{pop()}$}
            \State{$I_j \gets$ the left convex interval of $l_{bad}$}
            \State{$l \gets \operatorname{\bitan}_p(I_j, I_{i+1})$}
            \EndWhile
            \State{$S.\texttt{push(}l\texttt{)}$}
        \EndFor
        \State{\Return{$S$}}
    \end{algorithmic}
\end{algorithm}

The following three technical lemmas will help establish the correctness of~\Cref{alg:graham}.

\begin{lemma}\label{lemma.intervals} 
Suppose $p$ is a univariate polynomial and $I_1, I_2, I_3\subseteq \R$ are disjoint compact intervals with $I_1\le I_2 \le I_3$ and such that $\slope(\bitan_p(I_1,I_2)) >  \slope(\bitan_p(I_2,I_3))$.  Then 
\begin{equation}\label{slopes}
\slope(\bitan_p(I_1,I_2)) > \slope(\bitan_p(I_1,I_3)) > \slope(\bitan_p(I_2,I_3))
\end{equation}
If in addition
\begin{equation}\label{eq.minorize}
\bitan_p(I_1,I_2) \le p \text{ on } \conv(I_1\cup I_2) \; \text{ and } \; \bitan_p(I_2,I_3) \le p  \text{ on }
\conv(I_2\cup I_3),
\end{equation}
then
\begin{equation}\label{slopes.glued}
\bitan_p(I_1,I_3) \le p \text{ on } \conv(I_1\cup I_3). 
\end{equation}
\end{lemma}
\begin{proof}
 Let $s_1 := \slope(\bitan_p(I_1,I_2))$ and $s_2 := \slope(\bitan_p(I_2,I_3))$.  Thus $p_{I_1}^*(s_1) = p_{I_2}^*(s_1)$ and $p_{I_2}^*(s_2) = p_{I_3}^*(s_2)$.
Since $s_1 > s_2$, Lemma~\ref{lemma:binary}(c) applied to $I_2,I_3$ and $s_1> s_2$ implies that
\[
p_{I_1}^*(s_1) = p_{I_2}^*(s_1) < p_{I_3}^*(s_1).
\]
Thus Lemma~\ref{lemma:binary}(c) again implies that $\bar s = \slope(\bitan_p(I_1,I_3))$, that is, the solution to $p_{I_1}^*(s) = p_{I_3}^*(s)$, must satisfy $\bar s < s_1$.

Similarly, since $s_1 > s_2$, Lemma~\ref{lemma:binary}(c) applied to $I_1,I_2$ and $s_2> s_1$ implies that
\[
p_{I_3}^*(s_2) = p_{I_2}^*(s_2) > p_{I_1}^*(s_2).
\]
Thus Lemma~\ref{lemma:binary}(c) again implies that $\bar s > s_2$. Therefore 
$ s_1 > \bar s > s_2$ and~\eqref{slopes} is established.

We next show~\eqref{slopes.glued} provided~\eqref{eq.minorize} holds.  To that end, let $x_1\in I_1$ be a root of $p-\bitan_p(I_1,I_2)$ in $I_1$ and 
$x_3\in I_3$ be a root of $p-\bitan_p(I_2,I_3)$ in $I_3$. Since $\bitan_p(I_1,I_3)\le p$ on $I_1\cup I_3$, it follows that
\[
\bitan_p(I_1,I_3)(x_1) \le p(x_1) = \bitan_p(I_1,I_2)(x_1) \]
and 
\[
\bitan_p(I_1,I_3)(x_3) \le p(x_3) = \bitan_p(I_2,I_3)(x_3).
\]
Since the functions $\bitan_p(I_1,I_2), \bitan_p(I_2,I_3), \bitan_p(I_1,I_3)$ are all affine, the previous two inequalities and~\eqref{slopes} imply that
\[\bitan_p(I_1,I_3)(x) \le \bitan_p(I_1,I_2)(x) \text{ 
for all } x \ge x_1\]
 and 
\[
\bitan_p(I_1,I_3)(x) \le 
\bitan_p(I_2,I_3)(x)\text{ 
for all } x\le x_3.
\]  
Inequality~\eqref{slopes.glued} follows by 
combining the latter two inequalities and~\eqref{eq.minorize}.

\end{proof}

\begin{lemma}\label{lemma.slopes} Suppose $p$ is a univariate polynomial and $I_1, I_2, I_3\subseteq \R$ are disjoint compact intervals with $I_1\le I_2 \le I_3$ such that $\slope(\bitan_p(I_1,I_2)) <  \slope(\bitan_p(I_2,I_3))$.
Let $x_1^L \in I_1, x_1^U\in I_2$ be roots of $p -\bitan_p(I_1,I_2)$ in $I_1$ and $I_2$ respectively,  and  let $x_2^L \in I_2, x_2^U\in I_3$ be roots of $p -\bitan_p(I_2,I_3)$ in $I_2$ and $I_3$ respectively. Then the following inequalities hold.  First,
\begin{equation}\label{first}
x_1^U \le x_2^L.
\end{equation}
Second, when $x_1^L > \min I_1$ it also holds that
\begin{equation}\label{second}
p'(x_1^L) \le \slope(\bitan_p(I_1,I_2)).
\end{equation}
Third, when $x_1^U < \max I_2$ it also holds that
\begin{equation}\label{third}
 p'(x_1^U) \ge \slope(\bitan_p(I_1,I_2)) .
\end{equation}
The analogous inequalities also hold for $x_2^L$ when $x_2^L > \min I_2$, and for $x_2^U$ when $x_2^U < \max I_3$.
\end{lemma} 
\begin{proof}
Let $s_1 =  \slope(\bitan_p(I_1,I_2))$ and $s_2 =  \slope(\bitan_p(I_2,I_3))$.  The construction of $\bitan_p(\cdot)$ implies that 
\[
x_1^U = \argmax\{s_1x-p(x): x\in I_2\}, \; x_2^L = \argmax\{s_2x-p(x): x\in I_2\}.
\]
Since $s_2 > s_1$, it follows that $(s_2-s_1)(x_1^U-x) > 0$ for all $x\in I_2$ with $x < x_1^U$.  Since
$x_1^U = \argmax\{s_1x-p(x): x\in I_2\}$, it follows that
 for all $x\in I_2$ with $x < x_1^U$
\[
s_2(x_1^U-x) > s_1(x_1^U-x) \ge p(x_1^U) - p(x) \Leftrightarrow s_2 x - p(x) < s_2 x_1^U - p(x_1^U).
\]
Therefore
\[
x_2^L = \argmax\{s_2x-p(x): x\in I_2\} \ge x_1^U
\]
and hence~\eqref{first} follows.

Next, suppose $x_1^L > \min I_1$.  The construction of $\bitan_p(\cdot)$ and the choice of $x_1^L$ imply that for all $x\in I_1$ with $x< x_1^L$ the following holds
\[
p(x_1^L) - p(x) \le \bitan_p(I_1,I_2)(x_1^L) - \bitan_p(I_1,I_2)(x) = \slope(\bitan_p(I_1,I_2))(x_1^L - x).
\]
Thus
\[
p'(x_1^L) = \lim_{x\uparrow x_1^L} \frac{p(x_1^L) - p(x)}{x_1^L - x} \le\slope(\bitan_p(I_1,I_2))
\]
and hence~\eqref{second} follows.

Finally, suppose $x_1^U < \max I_2$.  Again the construction of $\bitan_p(\cdot)$ and the choice of $x_2^U$ imply that for all $x\in I_2$ with $x> x_1^U$ the following holds
\[
p(x) - p(x_1^U)  \ge \bitan_p(I_1,I_2)(x_1^U) - \bitan_p(I_1,I_2)(x) = \slope(\bitan_p(I_1,I_2))(x - x_1^U).
\]
Thus
\[
p'(x_1^U) = \lim_{x\downarrow x_1^U} \frac{p(x) - p(x_1^U)}{x- x_1^U} \ge\slope(\bitan_p(I_1,I_2))
\]
and hence~\eqref{third} follows.

\end{proof}

\begin{lemma}\label{lemma.stack} At the end of the $i$-th iteration of the for loop in \Cref{alg:graham} the stack $S$ (from bottom to top) is as follows 
\[
S = \{\bitan_p(I_{i_t},I_{i_{t+1}}): t=1,\dots,r\}
\]
for some $r \ge 1$ and $0=i_1 < i_2 < \cdots < i_{r+1} = i+1$.  Furthermore, the following two properties hold.  First, for $t=1,\dots,r$
\begin{equation}\label{bitangent.minorizes}
\bitan_p(I_{i_t},I_{i_{t+1}}) \le p \text{ on } \conv(I_{i_t} \cup I_{i_{t+1}}).
\end{equation}
Second, if $r>1$ then for $t=1,\dots,r-1$
\begin{equation}\label{slope.monotone}
\slope(\bitan_p(I_{i_t},I_{i_{t+1}})) \le \slope(\bitan_p(I_{i_{t+1}},I_{i_{t+2}})).
\end{equation}
\end{lemma}
\begin{proof}
We proceed by induction on $i$.  

{\bf Case $i=0$.}  Since $S$ is initially empty at the beginning of the $0$-th iteration, the while loop is entirely skipped and at the end of the $0$-th iteration $S = \{\bitan_p(I_{0},I_{1})\}$.  The bitangent  $\bitan_p(I_{0},I_{1})$ satisfies~\eqref{bitangent.minorizes} by~\Cref{cor.bitangent}. Since $r=1$ in this initial iteration,~\eqref{slope.monotone} trivially holds.

{\bf Case $i \ge 1$.} We will assume the above statements hold at the end of the $(i-1)$-th iteration and rely on the following claim.

{\bf Claim.} 
Whenever \Cref{alg:graham} updates the variable $l$, if the stack $S$ is nonempty, then the left interval $I_j$ of $l = \bitan_p(I_{j},I_{i+1})$ coincides with the right interval of $\texttt{top}(S)$ immediately after such update, and $l \le p$ on $\conv(I_j \cup I_{i+1})$.

{\it Proof of the claim.} \Cref{alg:graham} updates the variable $l$ at two possible places: one at the beginning of the \texttt{for} loop and the other within the \texttt{while} loop.  We  prove the claim in each of them.  

At the beginning of the $i$-th iteration of the \texttt{for} loop the stack $S$ is exactly what it was at the end of the $(i-1)$-th iteration, namely:
\[
S = \{\bitan_p(I_{i_t},I_{i_{t+1}}): t=1,\dots,r\}
\]
for some $r \ge 1$ and $0=i_1 < i_2 \cdots < i_{r+1} = (i-1)+1$.  In particular 
$\texttt{top}(S)$ is of the form $\bitan_p(I_{j},I_{i})$ 
with right interval $I_i$  which coincides with the left interval of $l = \bitan_p(I_{i},I_{i+1})$.  In addition, \Cref{cor.bitangent} implies that $l \le p$ on $\conv(I_i \cup I_{i+1})$.

At the beginning of the \texttt{while} loop  the stack $S$ is what it was at the end of the $(i-1)$-th iteration after having removed some (but not all) of its top elements.  In particular, 
\[
S = \{\bitan_p(I_{i_t},I_{i_{t+1}}): t=1,\dots,r\}
\]
for some $r \ge 1$ and $0=i_1 < i_2 < \cdots < i_{r+1}\le i$ such that~\eqref{bitangent.minorizes} and~\eqref{slope.monotone} hold.  
In addition, at this stage (the beginning of the \texttt{while} loop) we must have $l= \bitan_p(I_{i_{r+1}},I_{i+1})$ and  $l\le p$ on $\conv(I_{i_{r+1}}\cup I_{i+1})$ as this must have held the last time $l$ was updated. 

Next observe that if at the end of the \texttt{while} loop $S$ is non-empty then it must be the case that $r > 1$ and at that place in the \texttt{while} loop we must have
\[
l_{\text{bad}} = \bitan_p(I_{i_r},I_{i_{r+1}}), S = \{\bitan_p(I_{i_t},I_{i_{t+1}}): t=1,\dots,r-1\}, l = \bitan_p(I_{i_r},I_{i+1}).
\]
Hence $\texttt{top}(S)=\bitan_p(I_{i_{r-1}},I_{i_{r}})$ so its right interval evidently coincides with the left interval of $l = \bitan_p(I_{i_r},I_{i+1})$.  Furthermore, Lemma~\ref{lemma.intervals} applied to $I_1 = I_{i_{r}}, I_2 = I_{i_{r+1}}, I_3 = I_{i+1}$ implies that $l = \bitan_p(I_{i_r},I_{i+1}) \le p$ on $\conv(I_{i_r}\cup I_{i+1})$.

This concludes the proof of the claim.

To finish the proof of~\Cref{lemma.stack}, we consider two possible cases that can occur  right before the update on $S$ in the last step of the for loop.

{\bf Case 1.} The stack $S$ is empty.  In this case, the 
the last element popped out of $S$ must have been of the form $l_{\text{bad}} = \bitan(I_0,I_{j})$. Hence the above claim implies that $l = \bitan(I_0,I_{i+1})$ with $l = \bitan(I_0,I_{i+1}) \le p$.  Thus after the update at the end of the $i$-th iteration of the \texttt{for} loop we have $S = \{\bitan(I_0,I_{i+1})\}$ with $\bitan(I_0,I_{i+1})\le p$ on $\conv(I_0\cup I_{i+1})$.  In other words the one element of the stack $S$ satisfies~\eqref{bitangent.minorizes} and~\eqref{slope.monotone} trivially holds as $r=1$.

{\bf Case 2.} The stack $S$ is nonempty. In this case, the above claim implies that
\[
S = \{\bitan_p(I_{i_t},I_{i_{t+1}}): t=1,\dots,r-1\}
\]
for some $r \ge 2$ and $0=i_1 < \cdots < i_{r}\le i$ such that~\eqref{bitangent.minorizes} and~\eqref{slope.monotone} hold, and also 
\[
l = \bitan_p(I_{i_{r}},I_{i+1}) \le p
\] 
on $\conv(I_{i_{r}}\cup I_{i+1})$. Furthermore, since we just exited the \texttt{while} loop, it must be the case that 
\[
\slope(\bitan_p(I_{i_{r-1}},I_{i_{r}})) < \slope(l) = 
\slope(\bitan_p(I_{i_{r}},I_{i+1})).
\]
Thus after the update at the end of the $i$-th iteration of the \texttt{for} loop we have 
\[
S = \{\bitan_p(I_{i_t},I_{i_{t+1}}): t=1,\dots,r\}
\]
for some $r \ge 2$ and $0=i_1 < \cdots < i_{r} < i_{r+1} = i+1$ such that that~\eqref{bitangent.minorizes} and~\eqref{slope.monotone} hold. 
\end{proof}

Theorem~\ref{thm:graham} below shows the most important technical result of this section, namely, the {\it Graham envelope} $e_{\cI} p$ defined below via~\eqref{graham.env}  is precisely the convex envelope of $p$ over $\cI$.

Suppose $p$ is a univariate polynomial $p$, and $\cI := [x_L, x_U] \subset \R$ is a compact interval.  Let 
\[CI_\cI(p) = \{I_0,\dots,I_k\}\] 
be the set of convex intervals of $p$ with respect to $\cI$ returned by~\Cref{alg:ci}
and  let 
\[
S= \{\bitan_p(I_{i_t}, I_{i_{t+1}}) : t=1,\dots,r\}
\]
be the stack of bitangents returned by  \Cref{alg:graham} for some $0=i_1<i_2<\dots<i_{r} < i_{r+1} = k$. 

For each $\bitan_p(I_{i_t}, I_{i_{t+1}})$ let $x_{t}^L,\ x_{t}^U$ be roots of $\bitan_p(I_{i_t}, I_{i_{t+1}})$ in $I_{i_t}$ and $I_{i_{t+1}}$ respectively. 
Lemma~\ref{lemma.slopes} implies that $x_t^U \le x_{t+1}^L$ for $t=1,\dots,r$. For convenience, let $x_0^R:=x_L$ and $x_U:=x_{r+1}^L$. Thus
\[
x_L=x_0^R 
\le x_1^L < x_1^U \le x_2^L < x_2^U \le \cdots \le x_{r}^L < x_{r}^U \le x_{r+1}^L =
x_U.
\]
We rely on these cutoffs to construct the  function $e_{\cI} p: \cI\rightarrow \R$ piecewise as follows:
\begin{equation}\label{graham.env}
e_{\cI} p(x) = 
\begin{cases}
    \bitan_p(I_{i_t}, I_{i_{t+1}})(x) \; \text{ if }\ x_{t}^L  < x <  x_{t}^U \text{ for some } t \in \{1,\dots,r\}\\
    p(x) \; \text{ otherwise, that is, if}  \ x_{t-1}^U  \le x \le  x_{t}^L \text{ for some } t \in \{1,\dots,r+1\}.
\end{cases}
\end{equation}


\begin{theorem}[Correctness of \Cref{alg:graham}]\label{thm:graham}
    Let $p$ be a univariate polynomial $p$ of degree $n$ over a compact real interval $\cI = [x_L, x_U]$, where $|p'|\le \mathcal L$ over $\cI$.
Then the piecewise function $e_{\cI}p$ constructed via~\eqref{graham.env} is precisely the convex envelope of $p$ over $\cI$.       
    Furthermore, \Cref{alg:graham} terminates within $O(n)$ queries to $\bitan()$ defined by \Cref{alg:bisect}, plus an additional $O(n)$ time, with an overall time complexity of { $O\left(n^2\log(1/\epsilon)\log (\mathcal L/\epsilon)\right)$}.
\end{theorem}
\begin{proof}
{For runtime of this algorithm, note that each convex interval only enters and leaves the stack $S$ at most once each; since there are $O(n)$ convex intervals, there are $O(n)$ stack operations; and since there is one bitangent computation per stack operation, the overall runtime is dominated by $O(n)$ queries to \Cref{alg:bisect}, so the overall runtime is indeed $O\left(n^2\log(1/\epsilon)\log (\mathcal L/\epsilon)\right)$.}

Now we prove that $e_{\cI}p$ is the convex envelope of $p$ on $\cI$.  First, the construction~\eqref{graham.env} of $e_{\cI}p$ and~\Cref{lemma.stack} 
imply that $e_{\cI}p \le p$.  In addition, $e_{\cI}p$ is convex because Lemma~\ref{lemma.slopes} implies that $e_{\cI}p$ has a non-decreasing derivative on $\cI$.  Hence to finish it suffices to show that
for any convex $g$ that minorizes $p$ on $\cI$ the following inequality holds for all $x\in \cI$:
\begin{equation}\label{tight}
e_{\cI}p(x)\ge g(x).
\end{equation}
To that end, let $g$ be a convex minorizer of $p$ on $\cI$ and $x\in \cI$.  We will consider two possible cases:

{\bf Case 1:} The point $x$ satisfies $x_t^L < x < x_t^U$ for some $t\in\{1,\dots,r\}$.  Since $x_t^L,x_t^U$ are roots of $p- \bitan_p(I_{i_t},I_{i_{t+1}})$ and $\bitan_p(I_{i_t},I_{i_{t+1}})$ is affine, equation~\eqref{graham.env} implies that
\[
\begin{aligned}
e_{\cI}p(x) &= \bitan_p(I_{i_t},I_{i_{t+1}})(x) \\
&= \frac{x-x_t^L}{x_t^U-x_t^L} \bitan_p(I_{i_t},I_{i_{t+1}})(x_t^U) + \frac{x_t^U-x}{x_t^U-x_t^L} \bitan_p(I_{i_t},I_{i_{t+1}})(x_t^L) \\&= \frac{x-x_t^L}{x_t^U-x_t^L} p(x_t^U) + \frac{x_t^U-x}{x_t^U-x_t^L} p(x_t^L) \\&\ge \frac{x-x_t^L}{x_t^U-x_t^L} g(x_t^U) + \frac{x_t^U-x}{x_t^U-x_t^L} g(x_t^L) \\&\ge g(x).
\end{aligned}
\]
The first step above follows from~\eqref{graham.env}.  The second step holds because $\bitan_p(I_{i_t},I_{i_{t+1}})$ is affine and $x = \frac{x-x_t^L}{x_t^U-x_t^L}x^U_t+
\frac{x_t^U-x}{x_t^U-x_t^L}x^L_t.$ The third step holds because $x_t^L,x_t^U$ are roots of $p- \bitan_p(I_{i_t},I_{i_{t+1}})$. The fourth step holds because $g$ minorizes $p$.  The fifth step holds because $g$ is convex.

{\bf Case 2:} The point $x$ is outside all of the open intervals $(x_t^L,x_t^U)$ for $t\in\{1,\dots,r\}$.  In this case equation~\eqref{graham.env} implies that
\[
e_{\cI}p(x) = p(x)\ge g(x).
\]
In either case~\eqref{tight} holds.
\end{proof}

%% file: sections/gam.tex
\section{Convex Relaxation of  Polynomial Kolmogorov–Arnold Networks}
\label{sec:pkan} 

We focus on  {Polynomial Kolmogorov–Arnold Networks} (PKANs), a subclass of KANs \cite{liu2024kan} in which each univariate component is represented by a polynomial. This restriction preserves the expressive power of standard KANs while enabling exact convexification of each component. Using the Continuous Graham Scan (\Cref{alg:graham}), we construct exact convex envelopes for each polynomial and combine them according to the network’s additive structure to obtain a global convex relaxation. Formally, we define PKANs as follows:

\begin{definition}[Polynomial Kolmogorov-Arnold Network (PKAN)]\label{def:pkan}
    A Polynomial Kolmogorov-Arnold Network (PKAN) with $L$ layers is defined as 
    \begin{equation}\label{eq:kan2}
        \mathbf{PKAN}(\mathbf x) := \left(\Phi_L \circ \Phi_{L-1} \circ \cdots \circ \Phi_1\right)(\mathbf x)
    \end{equation}
    where each $\Phi_K: \R^{d_{K-1}} \to \R^{d_K}$ can be written component-wise as
    \begin{equation}\label{eq:kan-component}
    (\Phi_K(\mathbf x))_i = \sum_{j=1}^{d_{K-1}} \phi_{K, i, j}(x_j),
    \end{equation}
    with univariate \textbf{\underline{polynomials}} $\phi_{K,i,j}$ of degree $n$. The positive integers $d_K$ indicate the size of the $K^{\text{th}}$ layer of the KAN, \textbf{\underline{and $d_L := 1$.}}
\end{definition}

To enable global optimization of $\mathbf{PKAN}(\mathbf x)$ over a hyper-rectangle domain $\prod_{i=1}^{d_0} [l_i, u_i]$, with $l_i \le u_i$, we construct a high-quality convex relaxation of the network. This is done in two steps. First, we reformulate  \Cref{eq:kan2} as a constrained optimization problem via an epigraph representation, where each constraint involves a sum of univariate polynomials corresponding to a network layer. Second, we apply the Continuous Graham Scan (\Cref{alg:graham}) to each univariate polynomial to obtain its exact convex envelope and combine these envelopes according to the additive structure of the network, yielding a tractable global convex relaxation of $\mathbf{PKAN}(\cdot)$.

Specifically, 
epigraph formulation of the KAN objective is 
\begin{equation}\label{eq:kan-mp}
\begin{aligned}
\min_{z, t} \;& t \\
\text{s.t. } 
& t \ge z_{L,1}, \\
& z_{0,i} \in [l_i, u_i], && \forall \ i \in \{1,...,d_0\}, \\
& z_{K,i} = \sum_{j=1}^{d_{K-1}} \phi_{K,i,j}(z_{K-1,j}), 
&& \forall \ K \in \{1,...,N\},  \ i \in \{1,...,d_K\},
\end{aligned}
\end{equation}
and its the convex relaxation is
\begin{equation}\label{eq:kan-relax}
\begin{aligned}
\min_{z,t} \;& t \\
\text{s.t. } 
& t \ge z_{L,1}, \\
& z_{K,i} \in \left[l_{K}^{(i)}, u_{K}^{(i)}\right], 
&& \forall  \ K \in \{1,...,N\},  \ i \in \{1,...,d_i\}, \\
& z_{K,i} \ge \sum_{j=1}^{d_{K-1}} e_{[l_{K-1}^{(j)}, u_{K-1}^{(j)}]}\phi_{K,i,j}(z_{K-1,j}), 
&& \forall  \ K \in \{1,...,N\},  \ i \in \{1,...,d_K\}, \\
& z_{K,i} \le \sum_{j=1}^{d_{K-1}} E_{[l_{K-1}^{(j)}, u_{K-1}^{(j)}]}\phi_{K,i,j}(z_{K-1,j}), 
&& \forall  \ K \in \{1,...,N\},  \ i \in \{1,...,d_K\}.
\end{aligned}
\end{equation}

where $e_S f, E_S f$ denote the tight convex and concave envelopes of $f$ over a set $S$, and we obtain such envelopes using \Cref{alg:graham}.
We note that when the PKAN appears in the constraints (e.g., $\mathbf{PKAN}(x) \le 0$), the same construction yields a convex outer approximation of the feasible set. In this case, the envelope constraints relax the graph of the PKAN, ensuring that the feasible region of~\eqref{eq:kan-relax} contains that of~\eqref{eq:kan-mp}. Thus, the formulation remains unchanged, with the interpretation shifting from a lower bound on the objective to a valid relaxation of the constraint set.

The variable bounds $[l_{K}^{(i)}, u_{K}^{(i)}]$ can be recursively computed via
$$\begin{cases}
    [l_{0}^{(i)}, u_{0}^{(i)}] := [l_i, u_i] \\\\
    \displaystyle{[l_{K}^{(i)}, u_{K}^{(i)}] := \left[\sum_{j=1}^{d_{K-1}} \inf_{x\in [l_{K-1}^{(j)}, u_{K-1}^{(j)}]} \phi_{K,i,j}(x),\ \sum_{j=1}^{d_{K-1}} \sup_{x\in [l_{K-1}^{(j)}, u_{K-1}^{(j)}]} \phi_{K,i,j}(x)\right]}
\end{cases}$$
Note that computing $\inf$ and $\sup$ over fixed intervals is efficient: it suffices to first construct the convex and concave envelopes $e_{[l,u]}\phi$ and $E_{[l,u]}\phi$ using \Cref{alg:graham}, and then perform one-dimensional optimization either via bisection on the derivative or using the golden-section method on the envelopes themselves.
In the relaxed problem in \eqref{eq:kan-relax}, each constraint is a sum of univariate envelopes. Each envelope is piecewise, with segments that are either affine or known polynomials, which allows straightforward computation of second-order information. Because problem in \eqref{eq:kan-relax} is convex, it can be efficiently solved using standard interior-point solvers such as \texttt{Ipopt} \cite{wachter2006implementation}.

\subsection{Tight Convex Relaxation of Polynomial GAMs}
\label{sec:gams}
In this section, we establish a provable relaxation guarantee for a subclass of PKANs corresponding to Generalized Additive Models (GAMs) \cite{hastie1986generalized, hastie2017generalized}. Specifically, we show that the Continuous Graham Scan algorithm from \Cref{sect:graham} produces a convex relaxation for GAMs when all univariate components are polynomials and the link function $\Phi$ is monotone. This result illustrates how the tight univariate envelopes developed in \Cref{sect:graham} propagate through the additive structure of GAMs without introducing conservatism at the global optimum.

First introduced by Hastie and Tibshirani \cite{hastie1986generalized} and implemented in surrogate modeling software such as \texttt{PyGAM} \cite{serven2018pygam}, a GAM is a nonlinear function $M:\mathbb{R}^n \to \mathbb{R}$ of the form
\begin{equation}\label{eq:gam}
    M(x) := \Phi\left(\sum_{i} \psi_i(x_i)\right)
\end{equation}
where each $\psi_i$ is a univariate function and $\Phi$ is the link function that aggregates the component outputs into a single scalar. Throughout this section, we assume $\Phi$ is monotone (either increasing or decreasing) and invertible, which ensures that the convexity of the univariate envelopes is preserved in the overall relaxation.

Note that a GAM whose univariate functional forms are all polynomials can be viewed as a special case of a Polynomial KAN. Specifically, consider a KAN with two layers, $L=2$, and layer dimensions $[d_0, d_1, d_2] = [\operatorname{dim}(\mathbf{x}),1,1]$, so that
\[
\mathbf{K A} \mathbf{N}_{\mathrm{GAM}}(\mathbf{x})=\Phi_2 \circ \Phi_1(\mathbf{x}), \quad \Phi_1: \mathbb{R}^n \rightarrow \mathbb{R}, \quad \Phi_2: \mathbb{R} \rightarrow \mathbb{R}.
\]
If we define
\[
\Phi_1(\mathbf{x}):=\sum_{i} \psi_i\left(x_i\right), \quad \Phi_2:=\Phi,
\]
where $\Phi$ is the link function of the GAM, then $\mathbf{KAN}_{\text{GAM}}(\mathbf{x})$ coincides exactly with the GAM expression. This identification shows that polynomial GAMs inherit the additive univariate structure of PKANs, allowing the same convexification techniques to be applied.

We now focus on GAMs with polynomial components and show how the tight envelopes computed by the Continuous Graham Scan can be incorporated into a McCormick-style composition framework. Consider a GAM of the form $M(x) = \Phi\left(\sum_i p_i(x_i)\right)$ where each $p_i$ is a univariate polynomial and $\Phi:\mathbb{R}\to\mathbb{R}$ is a differentiable monotone link function. Using the exact convex envelopes of the $p_i$'s computed via \Cref{alg:graham}, the classical recursive relaxation scheme of McCormick \cite{mccormick1976computability}, together with bounds propagation through $\Phi$, produces a convex relaxation $M'$ of $M$ that is tight at the global optimum when $M$ is being minimized. Specifically, as shown in \Cref{thm:gam-envelope},
\begin{equation}
\min _{\mathbf{x}} M^{\prime}(\mathbf{x})=\min _{\mathbf{x}} M(\mathbf{x}) \quad \text { and } \quad \arg \min _{\mathbf{x}} M^{\prime}(\mathbf{x})=\arg \min _{\mathbf{x}} M(\mathbf{x}) .
\end{equation}
The construction proceeds in two steps: first, the separable polynomial sum is relaxed component-wise, and second, these bounds are propagated through the monotone link function to obtain the overall convex relaxation.

We begin by formalizing the model class and the domain over which the relaxation is constructed.

\begin{definition}[Monotone Polynomial GAM]\label{def:mpgam}
    { Let $x_1, \cdots, x_d$ be the decision variables for optimization, and let hyper-rectangle $B = \prod_{i=1}^d [x_i^L, x_i^U] \subset \R^d$ be the permissible domain of all decision variables $x_i$. We define a Monotone Polynomial Generalized Additive Model (MPGAM) to be a function $M: B \to \R$, such that
    $$M(x) = \Phi(P(x)) = \Phi\left(\sum_i^d p_i(x_i)\right)$$ 
    where the link function $\Phi: \R \to \R$ is differentiable, monotone, and either increasing or decreasing, and each $p_i$ is a univariate polynomial over $x_i$.}
\end{definition}

Next, we define the separable relaxation of the polynomial sum, which is the key object that allows the univariate envelopes to scale to higher dimensions.

\begin{definition}[Component-wise Envelopes]\label{def:comp-relax}
    { Let $P(x) := \sum_{i=1}^d p_i(x_i)$ be the sum of univariate polynomials, and let the domain of $p$ be $B = \prod_{i=1}^d [x_i^L, x_i^U]$. Then we define 
    $$\begin{cases}
        P^-(x) := \sum_i e_{[x_i^L, x_i^U]} p_i(x_i)\\
        P^+(x) := \sum_i E_{[x_i^L, x_i^U]} p_i(x_i)
    \end{cases}$$
    to be the \textbf{component-wise convex envelope} (in the case of $P^-$) and \textbf{component-wise concave envelope} (in the case of $P^+$) of $P(x)$. In essence, $P^-$ is the sum of the convex envelope of each $p_i$, and $P^+$ is the sum of concave envelope of each $p_i$.}
\end{definition}

The following theorem shows that this separable construction is sufficient to preserve global optimality when composed with a monotone link function.

\begin{theorem}\label{thm:gam-envelope}
    { Let $M$ be a monotone polynomial GAM defined in \Cref{def:mpgam} over the domain $B = \prod_{i=1}^d B_i = \prod_{i=1}^d [x_i^L, x_i^U]$, and let $M': B \to \R$ be defined as
    { 
    $$M'(x) := \begin{cases}
        e_I \Phi\left(P^-(x)\right) &\text{if $\Phi$ is monotonically increasing; }\\
        e_I \Phi\left(P^+(x)\right) &\text{if $\Phi$ is monotonically decreasing}
    \end{cases}$$
    }
    where $I = [\min_{x\in B} P^-(x), \max_{x\in B} P^+(x)]$.} Then the following properties hold: 
    \begin{enumerate}[(a)]
        \item $M'$ is convex relaxation of $M$: specifically, $M'$ is convex and $M'(x) \le M(x) \quad \forall x \in B$.
        \item $\argmin_{x \in B} M'(x) = \argmin_{x \in B} M(x)$.
        \item $\min_{x \in B} M'(x) = \min_{x \in B} M(x)$
    \end{enumerate}
\end{theorem}
\vspace{3mm}
{ 
\begin{remark}
    The relaxation $M'$ of $M$ in \Cref{thm:gam-envelope} is essentially the same as formulating $M$ as a constrained optimization problem in epigraph form, as in \Cref{eq:kan-mp}, and relaxing each constraint to obtain \Cref{eq:kan-relax}. 
\end{remark}
}

\begin{proof}[Proof of \Cref{thm:gam-envelope}]

    We prove items (a)-(c) separately.  In (a) and (b) we rely on the following claim.

    {\bf Claim.} If $\Phi$ is monotonically increasing, then its convex envelope $e_I \Phi$ is also monotonically increasing.  Likewise if 
   $\Phi$ is monotonically decreasing, then its convex envelope $e_I \Phi$ is also monotonically decreasing. 

   By symmetry, it suffices to prove the first part of the claim.  Note that $e_I \Phi$ is a piecewise function made up of sections of $\Phi$ and bitangent lines to the graph of $\Phi$, so $e_I \Phi$ is differentiable.  If $e_I \Phi$ is not monotonically increasing then $e_I \Phi'(x) \le 0$ for some $x$, then either by definition or by Mean Value Theorem, we can find some $z$ such that $\Phi'(z) \le 0$, which contradicts the assumption that $\Phi$ is monotonically increasing.

    \begin{enumerate}[(a)]
        \item To show $M'$ being a convex relaxation of $M$, we first focus on the case where $\Phi$ is monotonically increasing: for any $v,w \in \R^d$, we define $q_{v,w}(t) := P^-(v+tw)$; intuitively, $q_{v,w}$ is the restriction of $P^-$ onto a one-dimensional affine subset of $\R^d$, parameterized by $t$; since $P^-$ is a sum of convex functions and therefore convex, we have $q_{v,w}$ convex. 
        This reduction to one-dimensional affine slices allows us to verify convexity using standard composition rules.
        Since both $e_I \Phi$ and $q_{v,w}$ are convex, and $e_I \Phi$ is monotonically 
        increasing, we have $e_I \Phi \circ q_{v,w}$ convex in $t$. Since $v,w$ are arbitrary, it follows from zeroth order convexity that $M' := e_I \Phi \circ P^-$ is convex in $B$. Furthermore, for any $x \in B$, we have 
        $$P^-(x) = \sum_i e_{[x_i^L, x_i^U]} p_i(x_i) \le \sum_i p_i(x_i) = P(x)$$
        where the inequality follows from definition of convex envelope, so
        $$M'(x) = e_I \Phi(P^-(x)) \le \Phi(P^-(x)) \le \Phi(P(x))$$
        where the first inequality is due to $e_I \Phi$ being the tight convex envelope of $\Phi$ within $I$, which is within the possible range of values of $P^-$; and the second inequality is due to $P^- \le P$ and that $\Phi$ is monotonically increasing. 
    
        In case where $\Phi$ is monotonically decreasing, the argument is symmetric with convexity obtained from the composition of a convex decreasing function with a concave inner function, and underestimation following from $P \le P^+$.
        For convexity, we perform a restriction and notice that $f \circ g$ is convex if $f$ is convex and decreasing, and $g$ is concave. For underestimation, we notice that $P \le P^+$ and use the fact that $\Phi$ is monotone decreasing to derive
        $$M'(x) = e_I \Phi(P^+(x)) \le \Phi(P^+(x)) \le \Phi(P(x))$$

        \item {  To show exactness at the global optimum, namely $\argmin_x M'(x) = \argmin_x M(x)$, we still focus first on the case where $\Phi$ is monotonically increasing, since the proof for when $\Phi$ is monotonically decreasing is similar.}

        Since $e_I \Phi$ monotonically increasing, we have that
        \begin{align*}
            \argmin_{x \in B} M'(x) &= \argmin_{x \in B} P^-(x) \\
            &= \prod_{i=1}^d \argmin_{z \in [x_i^L, x_i^U]} \left(e_{[x_i^L, x_i^U] }p_i(z)\right)\\
            &= \prod_{i=1}^d \argmin_{z \in [x_i^L, x_i^U]} \left(p_i(z)\right)\\
            &= \argmin_{x \in B} P(x) \\
            &= \argmin_{x \in B} M(x).
        \end{align*}
        The separability of $P^-$ implies that its minimization decomposes component-wise. Thus, the first equality is by monotonicity of $e_I \Phi$, the second and fourth equality is by the fact that each component of the summation is independent, the third equality is by definition of tight convex envelope, and the last equality is from recalling that $M = \Phi \circ P$ and that $\Phi$ is monotone. Throughout the equalities above, we use $\prod$ for cartesian product of an indexed family of sets. {  We therefore conclude that $\argmin_{x \in B} M'(x) = \argmin_{x \in B} M(x)$.

        In case where $\Phi$ is monotonically decreasing, it similarly holds that $E_I \Phi$ is also monotonically decreasing, so a similar chain of reasoning yields 
        \begin{align*}
            \argmin_{x \in B} M'(x) &= \argmax_{x \in B} P^+(x) \\
            &= \prod_{i=1}^d \argmax_{z \in [x_i^L, x_i^U]} \left(E_{[x_i^L, x_i^U] }p_i(z)\right)\\
            &= \prod_{i=1}^d \argmax_{z \in [x_i^L, x_i^U]} \left(p_i(z)\right)\\
            &= \argmax_{x \in B} P(x) \\
            &= \argmin_{x \in B} M(x).
        \end{align*}
        where the first and last equalities are due to $\Phi$ and $E_I\Phi$ being monotonically decreasing. We therefore similarly conclude that $\argmin_{x \in B} M'(x) = \argmin_{x \in B} M(x)$.
        }

        \item Finally, we show that the optimal objective value is unchanged, namely $\min_x M'(x) = \min_x M(x)$. {  Once again, we first focus on the case where $\Phi$ is monotonically increasing. } Fix $x^* := \argmin_{x \in B} M(x)$. Then by definition of tight convex envelope, we have
        $$P^-(x^*) = \sum_i e_{[x_i^L, x_i^U]} p_i(x_i^*) = \sum_i p_i(x_i^*) = P(x^*)$$
        so we have $P^-(x^*) = P(x^*) = \inf_B P^-$; due to the monotonicity of $\Phi$ {  and the optimality of $\inf_B P^-$}, we have $e_I \Phi(\inf_B P^-) = \Phi(\inf_B P^-)$, so 
        $$\min_{x\in B} M'(x) = M'(x^*) := e_I \Phi(P^-(x^*)) = \Phi(P(x^*)) =: M(x^*) = \min_{x \in B} M(x)$$
        and we are done with the case for $\Phi$ being monotonically increasing. 
        
        For the other case, we follow a similar proof, { where we fix $x^* := \argmin_{x\in B} M(x)$ but instead obtain $P^+(x^*) = P(x^*) = \sup_B P^+$, which by monotonicity of $\Phi$ and the optimality of $\sup_B P^+$, implies $E_I \Phi(\sup_B P^+) = \Phi(\sup_B P^+)$, and hence $\min_{x \in B} M'(x) = \min_{x\in B} M(x)$ follows similarly. }  
    \end{enumerate}
\end{proof}

This establishes that tight univariate envelopes, such as the ones obtained by the Continuous Graham Scan, combined with monotonicity are sufficient to obtain a globally exact convex relaxation for MPGAMs.



%% file: sections/results_v2.tex
\section{Numerical Results} \label{sec:results}

We evaluate the quality of the proposed convex relaxation on a set of randomly generated PKAN instances.
The architectures of the corresponding models are varied along four structural dimensions: the number of hidden layers ($L$), the number of input variables, the number of hidden units per layer ($d_i$), and the degree of the polynomials ($n$), each taking values in $\{4,5,6\}$. For simplicity, all hidden layers in a given architecture use the same number of units $d_i$.
We restrict the study to these values because larger configurations, particularly those involving $L=7$ or $d_i=7$, rapidly become computationally intractable to solve to global optimality.
For each architecture, 50 independent networks are randomly generated with all input variables bounded in the hypercube $[-1.5,1.5]$. 
All models were implemented in \texttt{Pyomo} \cite{bynum2021pyomo}, and all computational experiments were conducted on a Linux machine running Ubuntu, equipped with eight Intel, Xeon Gold 6234 CPUs (3.30 GHz) and 1~TB of RAM. A total of eight hardware threads were used for all runs.

The global optimum value $f^\star$ of every PKAN is computed by solving the optimization problem \eqref{eq:kan-mp} to global optimality with \texttt{SCIP v9.2.4} \cite{bolusani2024scip}. To assess the strength of the proposed relaxation, we compare it against the resulting lower bounds from the root-node relaxations of state-of-the-art global optimization solvers, namely \texttt{SCIP v9.2.4} \cite{bolusani2024scip}, \texttt{BARON v25.12.10} \cite{zhang2025solving}, and \texttt{LINDOGLOBAL v15.0} \cite{lin2009global}. These bounds correspond to the first convex relaxation constructed by each solver prior to any spatial branching. 
In our experiments, all solvers are run with their default settings, which include presolve procedures \cite{puranik2017deletion,achterberg2020presolve,matter2023presolving}, bounds tightening \cite{caprara2016theoretical,puranik2017domain,zhang2020optimality}, and advanced convexification strategies \cite{tawarmalani2001semidefinite, meyer2005convex,scott2011generalized}.
The objective value of this root-node relaxation provides a benchmark lower bound, which we denote by $f^{\mathrm{relax}}$.
For comparison, we construct the convex relaxation described in \eqref{eq:kan-relax} and compute its optimal objective value using \texttt{Ipopt v3.14.16} \cite{wachter2006implementation} to obtain our $f^{\mathrm{relax}}$. Despite the use of the solvers' default presolve and heuristic enhancements, the numerical results below demonstrate that the convex relaxation developed in this work produces  stronger lower bounds in competitive time for multi-input PKAN instances.

We evaluate the relaxations in terms of both computational time and tightness of the resulting lower bound. The reported computational time measures the effort required to both constructing and solving each relaxation. 
To assess relaxation tightness, we compute the relative optimality gap with respect to the global optimum,
$$
\text {Relative Gap}=\frac{|f^{\text{relax}}-f^{\star}|}{ \left|f^{\star}\right|+\epsilon} \cdot 100,
$$
where $\epsilon$ is a small numerical tolerance (e.g., $10^{-12}$). Smaller values indicate tighter relaxations.

For each architecture, we compute the sample mean of both the relative gap and the computational time across the 50 instances. We report $95\%$ confidence intervals constructed from the sample standard error using a Student-$t$ distribution. To facilitate comparisons across several orders of magnitude and highlight substantial differences in relaxation strength and computational cost, both quantities are displayed on logarithmic scales in the figures.

The bar plots report the sample means together with the corresponding confidence intervals. Additionally, each bar may include an integer label indicating the number of instances (out of 50) for which the corresponding solver failed to construct a valid root-node relaxation. In these cases, the solver terminated the root-node processing with a lower bound of negative infinity, requiring spatial branching to obtain a nontrivial bound. When such failures occur, the reported average gap and computational time are computed only over the successful instances. Consequently, bars with larger failure counts are based on smaller sample sizes and therefore tend to exhibit wider confidence intervals. If a bar does not display a label, this indicates that all 50 instances produced valid root-node relaxations and no failures were observed.

Figures~\ref{fig:hidden_dim_vs_layers_combined} and \ref{fig:poly_vs_input_combined} examine how the strength and computational cost of the proposed relaxation scale with increasing architectural complexity. The two experiments explore complementary dimensions of the PKAN design space. Figure~\ref{fig:hidden_dim_vs_layers_combined} varies the network width and depth while fixing the input dimension and polynomial degree at their largest values considered in the benchmark. In contrast, Figure~\ref{fig:poly_vs_input_combined} fixes the architecture to its largest configuration ($L=d_i=6$) and varies the input dimension and polynomial degree, thereby isolating the effect of increasing dimensionality and nonlinearity.
Together, these experiments assess how the proposed relaxation behaves as the structural complexity of the underlying PKAN grows, and compare its performance against the root-node relaxations produced by the global optimization solvers \texttt{BARON}, \texttt{SCIP}, and \texttt{LINDOGLOBAL}.

\begin{figure}[htbp]
\centering
\begin{subfigure}{\linewidth}
    \centering
    \includegraphics[width=\linewidth]{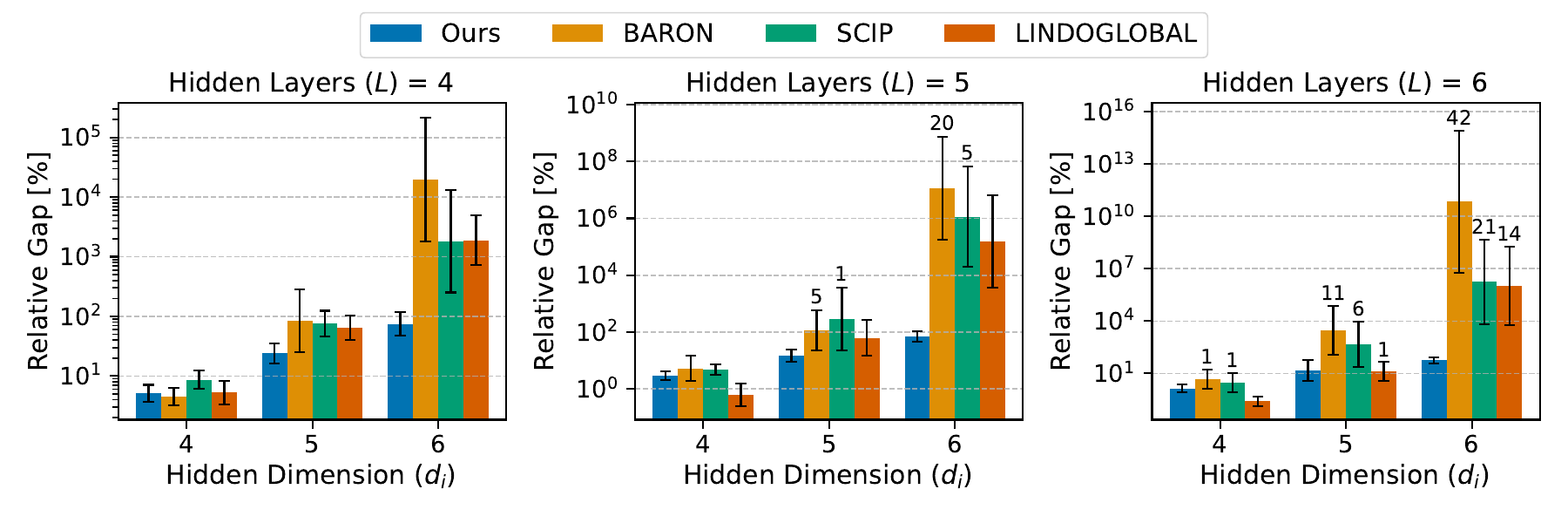}
    \caption{Relative optimality gap of the relaxation with respect to the global optimum.}
    \label{fig:gap_vs_hidden_dim}
\end{subfigure}
\vspace{0.4cm}
\begin{subfigure}{\linewidth}
    \centering
    \includegraphics[width=\linewidth]{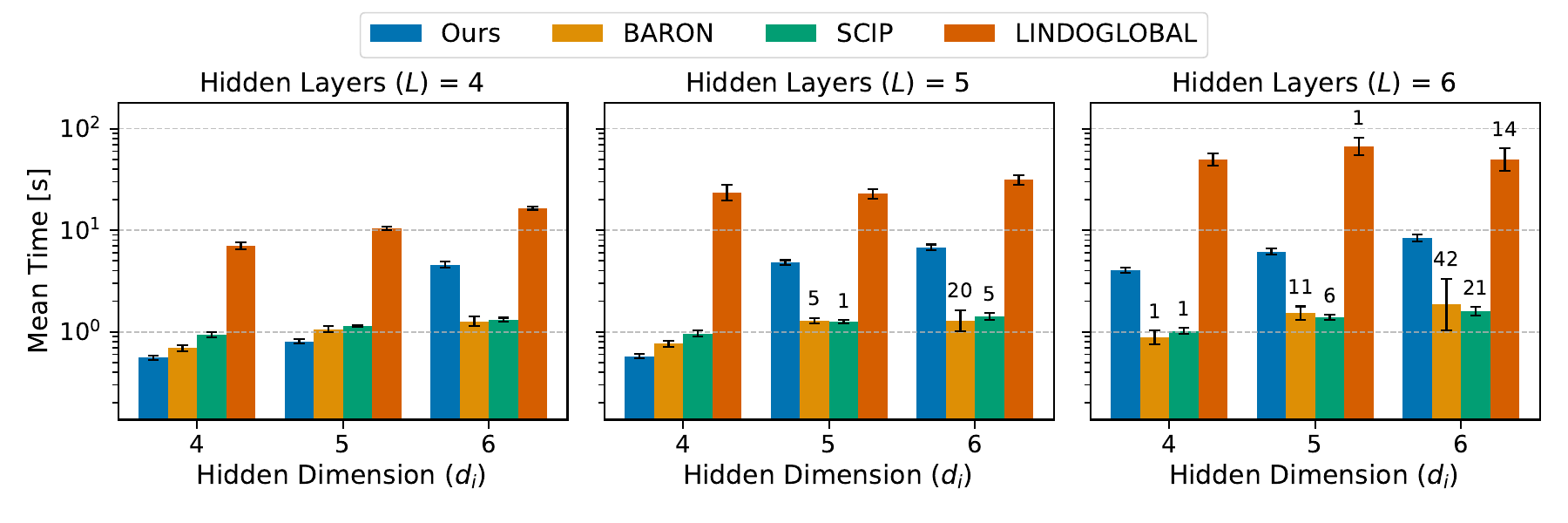}
    \caption{Mean computational time required to construct and solve each relaxation.}
    \label{fig:time_vs_hidden_dim}
\end{subfigure}
\caption{Performance of convex relaxations for PKAN models with six input variables and polynomial degree six as a function of the number of hidden units per layer, with panels grouped by the number of hidden layers. Bars report sample means across 50 randomly generated instances and error bars denote 95\% confidence intervals; both quantities are shown on logarithmic scales. Integer labels above bars indicate the number of instances (out of 50) for which the corresponding solver failed to produce a nontrivial root-node relaxation. When failures occur, averages are computed only over successful instances. Absence of a number indicates that all instances produced valid bounds.}
\label{fig:hidden_dim_vs_layers_combined}
\end{figure}

\begin{figure}[htbp]
\centering
\begin{subfigure}{\linewidth}
    \centering
    \includegraphics[width=\linewidth]{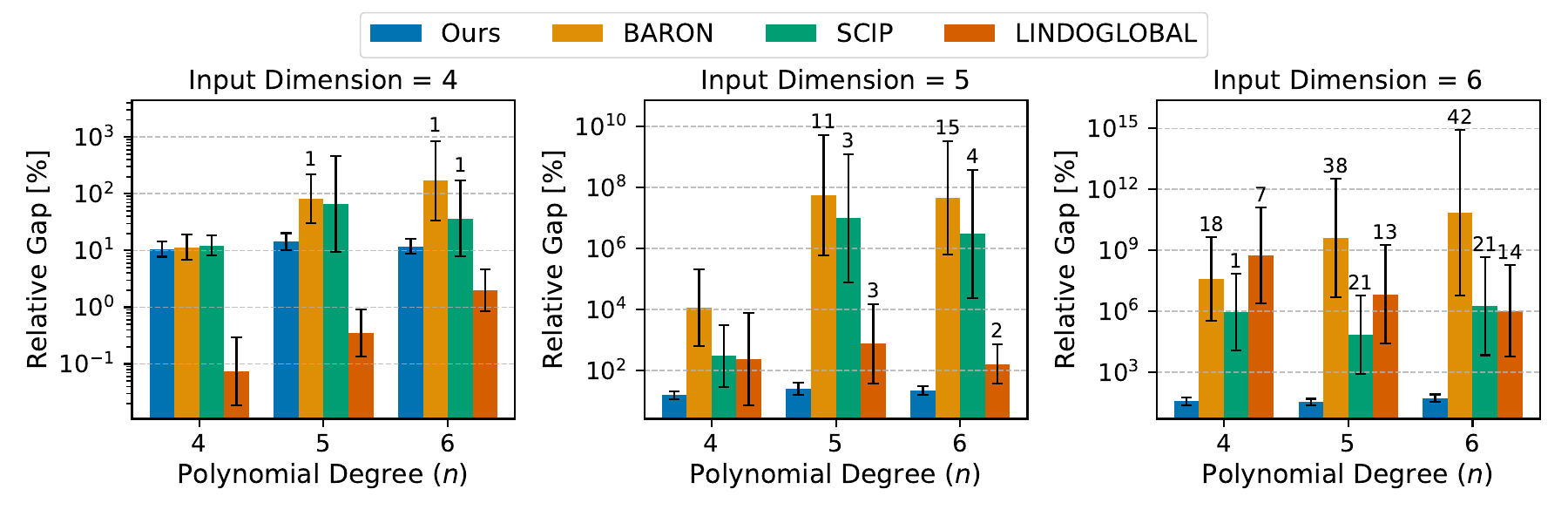}
    \caption{Relative optimality gap with respect to the global optimum as the polynomial degree and input dimension increase.}
    \label{fig:gap_vs_poly_deg_faceted_by_input_dim}
\end{subfigure}
\vspace{0.4cm}
\begin{subfigure}{\linewidth}
    \centering
    \includegraphics[width=\linewidth]{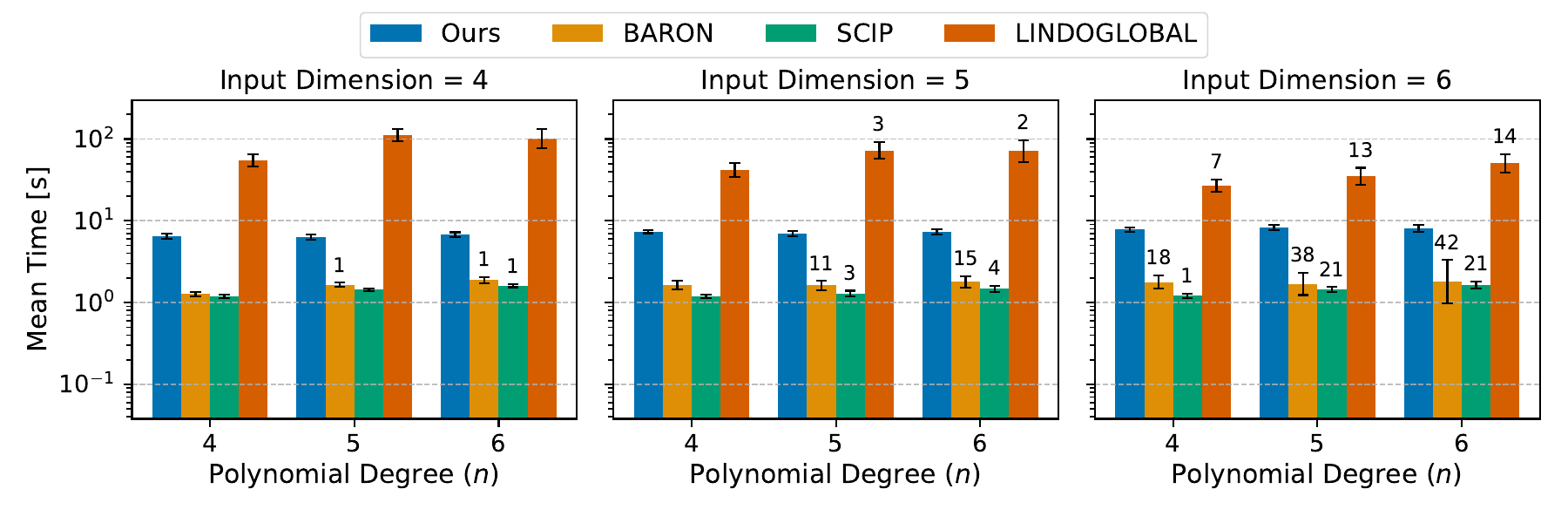}
    \caption{Mean time required to construct and solve the relaxation.}
    \label{fig:time_vs_poly_deg_faceted_by_input_dim}
\end{subfigure}
\caption{Effect of polynomial degree and input dimensionality on the quality and computational cost of convex relaxations for PKAN models with six hidden layers and six hidden units per layer. Bars represent averages over 50 randomly generated instances and error bars correspond to 95\% confidence intervals computed from the sample standard error. Both metrics are displayed on logarithmic scales. Integer labels above bars indicate the number of instances (out of 50) for which the solver failed to produce a valid root-node relaxation. When such failures occur, the reported statistics are computed only over successful instances. Absence of a label shows that all instances produced nontrivial bounds.}
\label{fig:poly_vs_input_combined}
\end{figure}

A first observation concerns the tightness of the resulting relaxations. 
For relatively narrow architectures ($d_i=4$), the proposed relaxation produces bounds that are comparable in strength to those obtained by the global optimization solvers. In these cases the gaps obtained by our method are similar to those of \texttt{SCIP} and \texttt{BARON}, while \texttt{LINDOGLOBAL} occasionally attaining slightly tighter bounds, typically within one order of magnitude.
As the network width increases, however, the advantage of the proposed relaxation becomes more pronounced. For architectures with five or six hidden units per layer, our approach consistently yields the tightest relaxations across all depths. In the widest configuration ($d_i=6$), the resulting bounds are between one and five orders of magnitude tighter than those produced by the competing solvers, with statistically significant differences across all values of $L$.

A second notable feature is the stability of the proposed relaxation with respect to both input dimensionality and polynomial degree, as illustrated in Figure~\ref{fig:poly_vs_input_combined}. Across all configurations, the relative optimality gaps produced by our method remain concentrated around the same order of magnitude, typically near the $10\%$ range. This behavior reflects the structure of the relaxation, which exploits the separability of PKAN layers by constructing convex envelopes for the underlying univariate functions. Consequently, the tightness of the relaxation depends primarily on the quality of these univariate envelopes and is largely insensitive to increases in the input dimension or polynomial degree.
In contrast, global optimization solvers construct relaxations directly in the multivariate space, which causes the quality of the relaxation to deteriorate rapidly as the dimensionality and degree of nonlinearity increase.

Another important distinction concerns robustness. As the architectures become larger and the number of dimensions grows, the global optimization solvers increasingly fail to construct valid root-node relaxations. In these cases the solvers return a lower bound of minus infinity and must rely on spatial branching to obtain their first nontrivial bound. This phenomenon becomes particularly pronounced in the largest instances. For the deepest and widest networks in Figure~\ref{fig:hidden_dim_vs_layers_combined}, \texttt{LINDOGLOBAL}, \texttt{SCIP}, and \texttt{BARON} generate valid root-node bounds for only approximately 64\%, 58\%, and 16\% of the instances, respectively. A similar deterioration is observed in Figure~\ref{fig:poly_vs_input_combined} as both input dimension and polynomial degree increase. In contrast, the proposed relaxation consistently produces finite and nontrivial bounds for all instances across every configuration considered.

Finally, the computational time results reveal a consistent trade-off between relaxation strength and solution time. For smaller architectures the proposed relaxation is typically the fastest to compute. As the networks grow larger, its computational cost increases but remains competitive, generally lying between the times required by \texttt{SCIP}/\texttt{BARON} and those of \texttt{LINDOGLOBAL}. Importantly, whenever the proposed relaxation is slower than \texttt{SCIP} or \texttt{BARON}, it simultaneously produces substantially tighter bounds, often by several orders of magnitude. Moreover, across all configurations the proposed approach remains consistently at least one order of magnitude faster than \texttt{LINDOGLOBAL}. These results highlight a favorable strength-to-cost trade-off for the proposed relaxation.

It is particularly noteworthy that the computational times of the proposed method remain comparable to those of these mature global optimization solvers despite several inherent disadvantages. In particular, constructing the relaxation requires computing all roots of the associated univariate polynomials, introducing an additional preprocessing step. This computation, however, is performed only once per polynomial and is independent of subsequent domain reductions. Consequently, if the proposed relaxation were embedded within a spatial branch-and-bound framework, these root computations could be reused across nodes, thereby reducing the amortized cost of constructing relaxations after branching. Furthermore, the competing solvers are implemented in highly optimized compiled languages such as C++ and Fortran and have benefited from years of engineering and performance tuning, whereas the current implementation of the proposed approach is written in Python. Taken together, these observations further highlight the favorable strength-to-cost trade-off achieved by the proposed relaxation.

Overall, the results indicate that exploiting the separable structure of PKAN layers leads to relaxations that are simultaneously stronger, more reliable, and computationally competitive with those generated by state-of-the-art global optimization solvers.

\section{Conclusions} \label{sec:conclusions}

This work addresses the challenge of constructing strong convex relaxations for polynomial Kolmogorov–Arnold Networks. Although the additive structure of KANs provides a promising foundation for tractable relaxations, existing approaches typically rely on generic factorable relaxations or discrete reformulations, which either fail to exploit the model structure or introduce unnecessary complexity. By restricting the univariate components to polynomials, we show that relaxing a PKAN reduces to computing convex envelopes of univariate polynomials. This observation transforms the relaxation of a high-dimensional nonconvex model into a collection of structured one-dimensional problems that can be treated analytically, yielding relaxations that directly leverage the separable architecture of KAN layers.

To solve this subproblem, we introduced the Continuous Graham Scan, a continuous convexification algorithm that constructs the exact convex envelope of a univariate polynomial over a bounded interval. The method exploits the geometry of the polynomial directly, combining a characterization of its convexity regions with an efficient and provably correct procedure for computing the unique bitangents that define the envelope.
We established the correctness of the algorithm and analyzed its computational complexity. In particular, for a polynomial $p$ of degree $n$, the method requires 
$O\left(n^2 \log (1 / \epsilon) \log (\mathcal L / \epsilon)\right)$ operations where $\epsilon$ denotes the desired numerical precision and $\mathcal L$ is a bound on the derivative satisfying $|p'|\le \mathcal L$. 
Leveraging this algorithm, we construct strong convex relaxations for PKANs by computing exact envelopes of their univariate polynomial components. Furthermore, we show that for the special case of Generalized Additive Models with a monotone link function, the resulting relaxation is tight at the global optimum.

Overall, the computational experiments demonstrate that exploiting the separable structure of PKAN layers yields relaxations that are both strong and robust. For narrow architectures, the proposed relaxation achieves bounds comparable to those produced by state-of-the-art global optimization solvers. As network width increases, however, the advantage becomes increasingly pronounced, with the proposed method producing bounds that are up to several orders of magnitude tighter. At the same time, the tightness of the relaxation remains remarkably stable with respect to increases in input dimensionality and polynomial degree, in contrast to multivariate relaxations constructed by general-purpose solvers, whose quality deteriorates rapidly as dimensionality and nonlinearity grow. The proposed approach also exhibits significantly greater robustness: while global solvers frequently fail to produce finite root-node bounds for the largest instances, the proposed relaxation consistently yields valid lower bounds across all tested configurations. Finally, these improvements in bound quality are achieved at a computational cost that remains competitive with mature global optimization solvers, highlighting a favorable strength–cost trade-off for the proposed relaxation.

A final practical observation concerns the cost of constructing the proposed relaxations. The most computationally demanding step is the initial computation of the polynomial roots required by the envelope construction. However, this step only needs to be performed once for each polynomial, since the roots are intrinsic to the polynomial itself and do not change when the variable bounds are updated. This is particularly relevant in global optimization algorithms such as spatial branch-and-bound, where interval bounds are repeatedly tightened throughout the search. As a result, the one-time root-computation cost can be amortized over many relaxation updates, making the effective cost of recalculating the relaxation at new bounds significantly lower. Incorporating the proposed relaxations into full spatial branch-and-bound frameworks and evaluating their impact on overall algorithmic performance, remains an interesting direction for future research.
